\documentclass[11pt]{amsart}
\usepackage[usenames]{color}
\usepackage{fullpage}
\usepackage{amscd}
\usepackage{amssymb, latexsym}
\usepackage[all,2cell,ps]{xy}
\usepackage{mathdots}

\catcode`Œ=\active \def Œ{\'{S}}

\theoremstyle{plain}
\newtheorem{thm}{Theorem}[section]
\newtheorem{theorem}[thm]{Theorem}
\newtheorem{mthm}[thm]{Main Theorem}
\newtheorem{cor}[thm]{Corollary}
\newtheorem{corollary}[thm]{Corollary}
\newtheorem{lm}[thm]{Lemma}
\newtheorem{lemma}[thm]{Lemma}

\newtheorem{prop}[thm]{Proposition}
\newtheorem{proposition}[thm]{Proposition}

\theoremstyle{definition}
\newtheorem{de}[thm]{Definition}
\newtheorem{exm}[thm]{Example}
\newtheorem{rem}[thm]{Remark}
\newtheorem{remark}[thm]{Remark}
\newtheorem{example}[thm]{Example}

\newcommand{\Z}{\mathbb{Z}}
\newcommand{\A}{\mathcal{A}}
\newcommand{\im}{\mathrm{Im}}

\newcommand{\zie}[1]{{\color{green}{#1}\color{black}{}}}

\newcommand{\aut}[1]{\mathrm{Aut}(#1)}
\newcommand{\dis}[1]{\mathrm{Dis}(#1)}
\newcommand{\lmlt}[1]{\mathrm{LMlt}(#1)}

\newcommand{\ld}{\backslash}
\newcommand{\orb}[2]{\mathrm{Orb}_{#1}(#2)}
\newcommand{\Ker}{\mathop{\mathrm{Ker}}}
\newcommand{\siq}{\mathop{\mathcal{SIQ}}}

\numberwithin{equation}{section}

\begin{document}

\title{Subdirectly irreducible medial quandles}

\author{P\v remysl Jedli\v cka}
\author{Agata Pilitowska}
\author{Anna Zamojska-Dzienio}

\address{(P.J.) Department of Mathematics, Faculty of Engineering, Czech University of Life Sciences, Kam\'yck\'a 129, 16521 Praha 6, Czech Republic}
\address{(A.P., A.Z.) Faculty of Mathematics and Information Science, Warsaw University of Technology, Koszykowa 75, 00-662 Warsaw, Poland}

\email{(P.J.) jedlickap@tf.czu.cz}
\email{(A.P.) apili@mini.pw.edu.pl}
\email{(A.Z.) A.Zamojska-Dzienio@mini.pw.edu.pl}

\thanks{While working on this paper, the second and third authors were supported by the
Statutory Grant of Warsaw University of Technology
504/02476/1120.}

\keywords{Quandles, medial quandles, modes, reductive algebras, subdirectly irreducible algebras.}
\subjclass[2010]{Primary: 08B26, 15A78. Secondary: 08A30, 03C05, 08A05, 20N02.}

\date{\today}

\begin{abstract}
We describe all subdirectly irreducible medial quandles. We show that they fall within one of four disjoint classes. In particular, in the finite case they are either connected
(and therefore Alexander quandles) or reductive. Moreover, we provide a representation of all non-connected subdirectly
irreducible medial quandles.
\end{abstract}

\maketitle

\section{Introduction}\label{sec1}
An algebra $(Q,\ast,\ld)$, with two basic binary operations $\ast$ and $\ld$, is called a \emph{quandle} if the
following conditions hold, for every $x,y,z\in Q$:
\begin{enumerate}
\item $x\ast x=x$,
\item $x\ast(y\ast z)=(x\ast y)\ast(x\ast z)$,
\item $x\ld (x\ast y)=y=x\ast (x\ld y)$.
\end{enumerate}
The operations $\ast$ and $\ld$ are called: \emph{multiplication} and \emph{left division}, respectively.  Condition $(1)$ says that multiplication is \emph{idempotent} and $(2)$ says it is \emph{left distributive}. Conditions $(3)$ define \emph{left quasigroup} property i.e., the equation $x\ast u=y$ has a unique solution $u\in Q$. It easily follows that left division is idempotent, too.
A quandle $Q$ is called \emph{medial} if, for every $x,y,u,v\in Q$,
\begin{align*}
(x\ast y)\ast(u\ast v)&=(x\ast u)\ast(y\ast v),\\
(x\ld y)\ld(u\ld v)&=(x\ld u)\ld(y\ld v),\\
(x\ld y)\ast(u\ld v)&=(x\ast u)\ld(y\ast v).
\end{align*}
Other names for the above property include: entropicity, bi-commutativity, alternation, bisymmetry, and abelianity.
Clearly, the class of all medial quandles forms a variety. A prototypic example of medial quandles is the class of \emph{Alexander quandles}: given a left $\Z[t,t^{-1}]$-module $A$, one defines the quandle over the set $A$ with the operations
\begin{align*}
x*y=(1-t)\cdot x+t\cdot y\;\;\text{and}\;\; x\ld y=(1-t^{-1})\cdot x+t^{-1}\cdot y.
\end{align*}
Alternatively, Alexander quandles can be regarded as pairs $(A,f)$, where $(A,+)$ is an abelian group, $f$ its automorphism and operations are given by
\begin{align*}
x*y=x-f(x)+f(y)=(1-f)(x)+f(y)\;\;\text{and}\;\; x\ld y=(1-f^{-1})(x)+f^{-1}(y).
\end{align*}
The variety of medial quandles is generated by Alexander quandles \cite{JPZ17}. Following universal algebra terminology, quandles embeddable into Alexander quandles (as subreducts of modules) will be called \emph{quasi-affine}.

This paper continues the research on medial quandles we started in \cite{JPSZ}. The primary significance of the quandle laws is that they are algebraic interpretations of the Reidemeister moves in knot theory \cite{Joy}. One of the strongest and oldest knot invariant is the \emph{Alexander polynomial}. The simplest way to understand it is in terms of Alexander quandles \cite[Chapter 3]{EN}. Similar role in case of another famous invariant given by \emph{Jones-Conway polynomial} play \emph{Conway algebras} (medial left quasigroups) \cite{Przyt}. Furthermore, most of quandles with a small number of elements are medial quandles \cite{JPSZ}.

Here, our aim is to further develop the structure theory of medial quandles
and to apply the theory to classify all subdirectly irreducible medial quandles.

An algebra is called \emph{simple} if it has exactly two congruence relations (i.e., equivalence relations invariant with respect to the operations).
Finite simple quandles were classified independently in \cite{AG,J82b}. The only two-element simple quandle is a \emph{projection} quandle satisfying the identity $x\ast y\approx y$. Since the orbit decomposition provides a congruence, simple quandles with more than two elements must be connected, hence, in the medial case, they must be Alexander quandles \cite[Theorem 7.3]{HSV}. As a special case of the classification, we obtain that
a finite medial quandle $Q$ is simple if and only if $Q$ is the Alexander quandle $(\Z_p^k,M)$ where $p$ is a prime and $M$ is the companion matrix of an irreducible monic polynomial in $\mathbb F_p[t]$.

An algebra $A$ is called a \emph{subdirect product} of algebras $S_i$, $i\in I$, if it embeds into the direct product $\prod_{i\in I} S_i$ in a way that every projection $A\to S_i$ is onto. An algebra $S$ is called \emph{subdirectly irreducible} (SI) if it admits no non-trivial subdirect representation. Birkhoff's theorem says that every algebra in a variety $\mathcal V$ embeds in a subdirect product of SI algebras from $\mathcal V$. Therefore, knowledge of SI algebras in a given variety provides a powerful tool for representation.
An easy-to-use criterion of subdirect irreducibility is provided by the following: an algebra $S$ is subdirectly irreducible if and only if the intersection of all non-trivial congruences, called the \emph{monolith congruence}, is non-trivial. In particular, every simple algebra is SI. See \cite[Section 3.3]{Ber} for details.

Here we summarize previous work on subdirectly irreducible medial quandles. On one hand, there are two papers which deal with special cases:
Roszkowska \cite{R99b} gave an explicit construction of all finite SI medial quandles that are involutory (2-symmetric), and Romanowska and Roszkowska \cite{RR89} did the same for the finite 2-reductive ones. Note that, in both papers, quandles were considered as algebras with one basic binary operation of multiplication; nevertheless the operation $\ld$ was implicitly present there anyway. It is well-known that for so-called $n$-symmetric (and therefore for all finite) quandles, the operation $\ld$ need not be taken into consideration since it is defined by means of multiplication (see e.g. \cite[Section 8.6]{RS} or \cite{St08}). In our previous paper \cite{JPSZ}, the operation $\ld$ was not taken into account, either. The reason was that
we worked with some defining structures rather than the operations themselves and
both the operations were actually implicitly present in the structure.

In a broader perspective, medial quandles are examples of \emph{modes} \cite{RS}, idempotent algebras
with a commutative clone of term operations. Kearnes \cite{K99} classifies SI modes according to the algebraic properties of blocks of their monolith. An SI mode $S$ has precisely one of the following three types:
\begin{itemize}
	\item the \emph{set type}: the clone of each monolith block is a clone of projections;
	\item the \emph{quasi-affine type}: the clone of each monolith block has either a cancellative binary operation or the clone is generated by $x+y+z\;\text{mod}\;2$;
	\item the \emph{semilattice type}: the clone of each monolith block has a non-cancellative essentially binary operation.
\end{itemize}
It is also shown that the SI mode is of the semilattice type if and only if it has a semilattice term \cite[Theorem 2.3]{K99}, and it is of the quasi-affine type if and only if each monolith block is a quasi-affine algebra which is not term equivalent to a set \cite[Theorem 2.12]{K99}.

The semilattice type is well understood \cite{K95} but cannot appear in quandles: the monolith of an SI semilattice mode $S$ has exactly one non-trivial block $M$ which consists of two special elements \cite[Theorem 3.1]{K95}. By idempotency, each block is a subalgebra of $S$. For quandles, $M$ is a subquandle of $S$ and left cancellativity implies that the clone of $M$ is a clone of projections.

Algebras of quasi-affine type are related to quasi-affine algebras
and questions about their properties are often reduced to module-theoretical questions.
SI medial quandles of quasi-affine type are presented in Section \ref{sec:si}.

In general, very little is known about set type SI modes.
Besides \cite{RR89}, a notable exception is \cite{St}, a classification of 2-reductive SI modes (they are all of set type).
Our paper fills partially the gap. We show that each SI medial quandle of the set type is either quasi-reductive or a two-element projection quandle.

The construction of all non-connected quasi-reductive SI medial quandles (Theorem \ref{thm1:subdir}) together with Main Theorem \ref{all_si} gives actually a complete description of all SI medial quandles.

It is also interesting to note that classification of SI racks (\emph{non-idempotent} quandles) uses substantially different techniques, see \cite{JKS,St08}.

The paper is organized as follows. In Section \ref{sec:decomposition theorem} some notions and results from \cite{JPSZ}
are recalled, in particular the representation of medial quandles as sums of affine meshes.
Each medial quandle can be constructed from abelian groups which are naturally equipped with the structure
of a $\Z[t,t^{-1}]$-module. In Section~\ref{sec:congruences on quandles} we describe a relationship between congruences
of a medial quandle that are below the orbit decomposition and submodules of the modules from which it is built.
Then, in Section \ref{sec:si}, we present SI medial quandles of quasi-affine type and we show that every finite SI medial quandle is either reductive or connected (thus an Alexander quandle).
In Section \ref{sec:reductivity} we develop more structure theory of reductive medial quandles.
In Section~\ref{sec:set-type} we analyze SI medial quandles of the set type. In particular, Main Theorem \ref{all_si} characterizes all SI medial quandles.
In Section \ref{sec:subdirectly irreducible finite reductive quandles} we give an explicit construction of non-connected
SI medial quandles of the set type (Theorem \ref{thm1:subdir}, Corollary \ref{cor:SI_red})
and also provide several examples of SI medial quandles.
In Section~\ref{sec:isomorphism} we tackle the question of isomorphisms between the quandles constructed
in the section before.
In Section~\ref{sec:classification} we present a classification of SI medial quandles in Theorem \ref{thm:all_class}. Next, we describe all infinite SI 2-reductive medial quandles (Theorem \ref{thm:all2red}) and all infinite SI involutory medial quandles (Theorem \ref{th:inf_inv}), which completes the classification given by Romanowska and Roszkowska.

\subsection*{Notation and basic terminology} The identity bijection
will always be denoted by 1 and the zero group homomorphism by $0$.
For two bijections $\alpha,\beta$, we write $^\beta\alpha=\beta\alpha\beta^{-1}$.
The commutator is defined as $[\beta,\alpha]=(^\beta\alpha)\alpha^{-1}$. If a group $G$ acts on a set $X$
then
the stabilizer of $e\in X$ will be denoted by~$G_e$.

When studying left quasigroups, as important tools we use the mappings $L_e\colon x\mapsto e\ast x$,
called the \emph{left translations}. We use also the right translations $R_e\colon x\mapsto x\ast e$.
The idempotency and the mediality imply that both $L_e$ and $R_e$ are endomorphisms.
The left quasigroup property means that $L_e$ is an automorphism. A quandle is called \emph{latin} (or a
\emph{quasigroup}), if the right translations, $R_e$, are bijective,
too.

We will often use the following observation: for every $a\in Q$ and $\alpha\in\aut Q$,
\begin{equation}
^\alpha L_a=L_{\alpha(a)}.\label{auto}
\end{equation}

A quandle $(Q,\ast,\ld)$ is $m$-\emph{reductive}, if it satisfies the identity
\begin{equation*}
(((x\ast\underbrace{y)\ast y)\ast\dots )\ast y}_{m-\text{times}}\approx y.
\end{equation*}
A quandle $Q$ is called \emph{reductive}, if it is $m$-reductive for some $m$ and $Q$ is \emph{strictly $m$-reductive} if it is $m$-reductive and not $k$-reductive for any $k<m$. In particular, every Alexander quandle $(A,1)$ is 1-reductive and is a projection quandle.

\section{Orbit decomposition}\label{sec:decomposition theorem}

In this section, we recall notions and results from \cite{JPSZ}
on representing medial quandles as sums of affine meshes.
We start with the definition of important permutation groups acting on~$Q$.

\begin{de}
The (left) \emph{multiplication group} of a quandle $Q$ is the
permutation group generated by left translations, i.e.,
\begin{equation*}
\lmlt Q=\langle L_a\mid a\in Q\rangle\leq \mathrm{Aut}(Q).
\end{equation*}
We define the \emph{displacement group} as the subgroup
\begin{equation*}
\dis Q=\langle L_aL_b^{-1}\mid a,b\in Q\rangle.
\end{equation*}
\end{de}
It follows that $\dis Q=\{L_{a_1}^{k_1}\dots L_{a_n}^{k_n}\mid a_1,\dots,a_n\in Q\text{ and }\sum_{i=1}^n k_i=0\}.$
Both groups act naturally on~$Q$ and
it was proved in~\cite[Proposition 2.1]{HSV} that $\lmlt{Q}$ and $\dis Q$ have the same orbits of action.
We refer to the orbits of transitivity of the groups $\lmlt Q$ and $\dis Q$ simply as \emph{the orbits of $Q$}, and denote $$Qe=\{\alpha(e)\mid\alpha\in\lmlt Q\}=\{\alpha(e)\mid\alpha\in\dis Q\}$$ the orbit containing an element $e\in Q$.
Notice that orbits are subquandles of $Q$.

One of the main results of~\cite{HSV} was that every orbit of a quandle~$Q$ admits a certain group
representation, called a \emph{homogeneous representation}, based on~$\dis Q$. In particular,
if $Q$ has only one orbit (such a quandle is called \emph{connected})
then the homogeneous representation based on $\dis{Q}$ is, in a sense, minimal
such a representation of the quandle~$Q$.

This article deals with medial quandles. From the group-theoretical point of view, the importance
of medial quandles comes from the fact that $\dis Q$ is abelian.

\begin{proposition}\cite[Section 1]{J82b}
 Let~$Q$ be a quandle. Then $Q$ is medial if and only if $\dis Q$ is commutative.
\end{proposition}

As we said, every orbit of a medial quandle~$Q$ admits a homogeneous representation
based on~$\dis Q$ and the fact that $\dis Q$ is abelian implies that the representation
actually reduces to the definition of an Alexander quandle.
The main result of~\cite{JPSZ} was a structural description of medial quandles
based on their orbits -- the tool we used to reconstruct the whole quandle
from its orbits was the affine mesh.

\begin{de}
An \emph{affine mesh} over a non-empty set $I$ is a triple
$$\mathcal A=((A_i)_{i\in I};\,(\varphi_{i,j})_{i,j\in I};\,(c_{i,j})_{i,j\in I})$$ where $A_i$ are abelian groups, $\varphi_{i,j}:A_i\to A_j$ homomorphisms, and $c_{i,j}\in A_j$ constants, satisfying the following conditions for every $i,j,j',k\in I$:
\begin{enumerate}
    \item[(M1)] $1-\varphi_{i,i}$ is an automorphism of $A_i$;
    \item[(M2)] $c_{i,i}=0$;
    \item[(M3)] $\varphi_{j,k}\varphi_{i,j}=\varphi_{j',k}\varphi_{i,j'}$, i.e., the following diagram commutes:
$$\begin{CD}
A_i @>\varphi_{i,j}>> A_j\\ @VV\varphi_{i,j'}V @VV\varphi_{j,k}V\\
A_{j'} @>\varphi_{j',k}>> A_k
\end{CD}$$
    \item[(M4)] $\varphi_{j,k}(c_{i,j})=\varphi_{k,k}(c_{i,k}-c_{j,k})$.
\end{enumerate}
If the index set is clear from the context, we shall write briefly $\mathcal A=(A_i;\varphi_{i,j};c_{i,j})$. Moreover, if $I$ is a finite set we will sometimes display an affine mesh as a triple $((A_i)_{i\in I};\,\Phi;C)$ where $\Phi=(\varphi_{i,j})_{i,j\in I}$ and $C=\,(c_{i,j})_{i,j\in I}$ are $\left|I\right|\times \left|I\right|$ matrices.
\end{de}

\begin{de}
The \emph{sum of an affine mesh} $(A_i;\varphi_{i,j};c_{i,j})$ over a set $I$ is an algebra defined on the disjoint union of the sets $A_i$,
with two operations
\begin{align*}
a\ast b&=c_{i,j}+\varphi_{i,j}(a)+(1-\varphi_{j,j})(b),\\
a\ld b&=(1-\varphi_{j,j})^{-1}(b-\varphi_{i,j}(a)-c_{i,j}),
\end{align*}
for every $a\in A_i$ and $b\in A_j$.
\end{de}

It was proved in \cite[Lemma 3.8]{JPSZ} that the sum of any affine mesh is a medial quandle.
Every summand $A_i$ becomes a subquandle of the sum, and for $a,b\in A_i$ we have
\begin{align*}
a\ast b&=\varphi_{i,i}(a)+(1-\varphi_{i,i})(b),\\
a\ld b&=(1-(1-\varphi_{i,i})^{-1})(a)+(1-\varphi_{i,i})^{-1}(b),
\end{align*}
hence $(A_i,\ast,\ld)$ is the Alexander quandle $(A_i,1-\varphi_{i,i})$.
Moreover, every summand turns out to be a union of orbits. If we want every summand to be a single orbit,
we have to add the indecomposability condition.

\begin{de}
An affine mesh $(A_i;\varphi_{i,j};c_{i,j})$ over a set $I$ is called \emph{indecomposable} if
$$A_j=\left<\bigcup_{i\in I}\left(c_{i,j}+\im(\varphi_{i,j})\right)\right>,$$ for every $j\in I$.
Equivalently, the group $A_j$ is generated by all the elements $c_{i,j}$, $\varphi_{i,j}(a)$ with $i\in I$ and $a\in A_i$.
\end{de}

\begin{theorem}\cite[Theorem 3.14]{JPSZ}\label{thm:decomposition}
An algebra is a medial quandle if and only if it is the sum of an indecomposable affine mesh. The orbits of the quandle coincide with the groups of the mesh.
\end{theorem}

Starting from a medial quandle~$Q$, a natural way to define an indecomposable affine mesh that
sums to~$Q$ is the canonical mesh.

\begin{de}
Let $Q$ be a medial quandle, and choose a transversal $E$ to the orbit decomposition. We define the \emph{canonical mesh} for $Q$ over the transversal $E$ as $\A_{Q,E}=(\orb Qe;\varphi_{e,f};c_{e,f})$ with $e,f\in E$ where for every $x\in Qe$
$$\varphi_{e,f}(x)=x\ast f-e\ast f\quad\text{and}\quad c_{e,f}=e\ast f.$$
\end{de}

\begin{lemma}\cite[Lemma 3.13]{JPSZ}
Let $Q$ be a medial quandle and $\A_{Q,E}$ its canonical mesh. Then $\A_{Q,E}$ is an indecomposable affine mesh and $Q$ is equal to its sum.
\end{lemma}

Alternatively, we could have defined the canonical mesh using the groups $A_e=\dis Q/\dis Q_e$, homomorphisms $\varphi_{e,f}(\alpha\dis Q_e)=[\alpha,L_e]\dis Q_f$, and constants $c_{e,f}=L_eL_f^{-1}\dis Q_f$. Then the original quandle $Q$ is \emph{isomorphic} to the sum of the mesh where the coset $\alpha\dis Q_e$ corresponds to the element $\alpha(e)\in Q$.

\section{Congruences below the orbit decomposition}\label{sec:congruences on quandles}

As it was shown in Theorem \ref{thm:decomposition}, each medial quandle is the sum of an indecomposable affine mesh $\mathcal A=(A_i;\varphi_{i,j};c_{i,j})$ over the index set $I$ and the orbits of the quandle coincide with the groups of the mesh. All the
abelian groups~$A_i$ can be naturally equipped with the structure of a $\Z[t,t^{-1}]$-module
by defining
$$t^n\cdot a=(1-\varphi_{i,i})^n (a),\qquad\text{for all }n\in\Z\text{ and }a\in A_i.$$
Moreover, we have
$$\varphi_{i,j} (t^n\cdot a)=\varphi_{i,j}(1-\varphi_{i,i})^n (a)=(1-\varphi_{j,j})^n\varphi_{i,j}(a)
=t^n \cdot\varphi_{i,j}(a),$$
for $n\geq 0$ and
$$\varphi_{i,j} (t^n\cdot a)=(1-\varphi_{j,j})^n(1-\varphi_{j,j})^{-n}\varphi_{i,j} (t^n\cdot a)
=(1-\varphi_{j,j})^n\varphi_{i,j}(1-\varphi_{i,i})^{-n}(1-\varphi_{i,i})^{n}(a)
=t^n \cdot\varphi_{i,j}(a),$$
for $n<0$. Therefore every $\varphi_{i,j}$ can be treated as a $\Z[t,t^{-1}]$-module homomorphism.
Hence, in the sequel, we shall assume that all the orbits are $R$-modules where~$R$
is a suitable image of $\Z[t,t^{-1}]$.

Let $Q$ be a medial quandle and $e\in Q$. Let $\alpha(e),\beta(e)\in
Qe$ with $\alpha,\beta\in\dis Q$ and put
$$\alpha(e)+\beta(e)=\alpha\beta(e),\quad -\alpha(e)=\alpha^{-1}(e),\quad
\text{and}\quad t^{n}\cdot\alpha(e)=\;^{L_{e}^{n}}\alpha(e)$$
Then $\orb Qe=(Qe,+,-,e,\cdot)$ is a $\Z[t,t^{-1}]$-module, called the \emph{orbit module} for $Qe$.

Let us note that the orbit decomposition provides a congruence, namely the relation $\pi\subseteq Q\times Q$ defined by
\begin{equation*}
a\mathrel{\pi} b\;\ {\rm iff}\;\; a=\alpha(b)\;\; {\rm for}\;\; {\rm some}\;\; \alpha\in \dis Q.
\end{equation*}
Clearly $\pi$ is an equivalence relation. Now let $a\mathrel{\pi} b$ and $c\mathrel{\pi} d$. Then $a=\alpha(b)$ and $c=\gamma(d)$ for some $\alpha,\gamma\in \dis Q$. By commutativity of the group $\dis Q$ we have
\begin{align*}
a\ast c=\alpha(b)\ast \gamma(d)=L_{\alpha(b)}\gamma(d)=\alpha L_b \alpha^{-1}\gamma(d)=\alpha L_b \alpha^{-1}\gamma L_d^{-1}(d)=\\
L_b \alpha^{-1}\gamma L_d^{-1}\alpha(d)=L_dL_d^{-1}L_b \alpha^{-1}\gamma L_d^{-1}\alpha L_d(d)=\\
L_d\alpha^{-1}\gamma L_d^{-1}\alpha L_d L_d^{-1}L_b(d)=
L_d\alpha^{-1}\gamma L_d^{-1}\alpha L_b(d)=L_d\alpha^{-1}\gamma L_d^{-1}\alpha (b\ast d),
\end{align*}
which shows that $(a\ast c) \mathrel{\pi} (b\ast d)$. Similar calculations show that $(a\ld c) \mathrel{\pi} (b\ld d)$ and
one obtains that $\pi$ is a quandle congruence.

\begin{prop}\label{lm1:congruences}
The relation $\pi$ is the least congruence on a quandle $Q$ such that the quotient $Q/{\pi}$ is the projection quandle.
\end{prop}

\begin{proof}
First note that, for any $a,b\in Q$, $a\ast b=L_aL_b^{-1}(b)$ and $a\ld b=L_a^{-1}L_b(b)$, which means that
$(a\ast b)\mathrel{\pi} b$ and $(a\ld b)\mathrel{\pi} b$. This shows that $Q/{\pi}$ is the projection
quandle.

Now, let $\psi$ be a congruence relation on $Q$ such that $Q/{\psi}$ is the projection quandle. Then, for any $x,y\in Q$,
$y\mathrel{\psi} (x\ast y)=L_x(y)$ and $y\mathrel{\psi} (x\ld y)=L_x^{-1}(y)$.

If $a\mathrel{\pi} b$ then $a=\alpha(b)$, for some $\alpha\in \dis Q$. By the definition of~$\dis Q$, $\alpha = L_{b_1}^{k_1}\dots L_{b_n}^{k_n}$, for some $b_1,\dots,b_n\in Q$, and $\sum_{i=1}^n k_i=0$. This gives the following:
\begin{equation*}
b\mathrel{\psi} L_{b_n}^{k_n}(b)\mathrel{\psi} L_{b_{n-1}}^{k_{n-1}}L_{b_n}^{k_n}(b)\mathrel{\psi}\ldots\mathrel{\psi} L_{b_1}^{k_1}\dots L_{b_n}^{k_n}(b)=\alpha(b)=a.
\end{equation*}
So, $a\mathrel{\psi} b$ and $\pi\subseteq\psi$.
\end{proof}

Now we will describe a relationship between congruences of a medial quandle $Q$ and congruences of the modules from which it is built. In particular, we will show that each congruence on~$Q$, when restricted to an orbit, is a module congruence.
In fact, one obtains a one-to-one correspondence between congruences on $Q$ below the congruence providing the orbit decomposition and families of submodules of orbit modules satisfying  an additional condition.

It is not usual to work with congruences of modules and we shall therefore be,
from now on, speaking about submodules instead of congruences of modules. In particular, if $\varrho$ is a congruence of a module $M$ then there is a submodule $M_{\varrho}$ of $M$ such that $a\mathrel{\varrho}b\mathrel{\Leftrightarrow} a-b\in M_{\varrho}$.
In the case if $a-b\in N$, for some submodule $N$, we will sometimes write $a\equiv_{N} b$.

\begin{thm}\label{thm:congruences}
Let $Q$ be a medial quandle being the sum of an affine mesh $(A_i;\,\varphi_{i,j};\,c_{i,j})$ over a set $I$.

Let $\varrho\subseteq \pi$ be a congruence relation on $Q$. Then, for each $i\in I$, $\varrho$ restricted to the orbit $A_i$, provides a $\Z[t,t^{-1}]$-submodule $M_i$ of the module $A_i$. Moreover, the following condition is satisfied
\begin{equation}\label{cond:lm3:congruences}
\varphi_{k,j}(M_k)\subseteq M_j,
\end{equation}
for each $k,j\in I$.

On the other hand, let $(M_i)_{i\in I}$ be a family of submodules such that each $M_i$ is a $\Z[t,t^{-1}]$-submodule of the module $A_i$ and the condition \eqref{cond:lm3:congruences} holds for each $k,j\in I$. Then the binary relation $\varrho\subseteq \pi$ defined for $a,b\in Q$ as
\begin{equation}\label{cond:lm2:congruences}
a\mathrel{\varrho} b \; {\rm iff}\;\; \exists (i\in I)\; a,b\in A_i\;\; {\rm and} \;\; a\equiv_{M_i} b,
\end{equation}
is a congruence relation on $Q$.
\end{thm}

\begin{proof}
Let $\mathcal A_{Q,E}$ be the canonical mesh of $Q$ and $\varrho$ be a congruence relation on $Q$ such that $\varrho\subseteq \pi$.
Let for some $x,y\in Q$, $x\mathrel{\varrho} y$. Since $\varrho$ is a quandle congruence, it follows that for each $z\in Q$,
\begin{equation*}
z\ast x=L_z(x)\mathrel{\varrho} L_z(y)=z\ast y,
\end{equation*}
and
\begin{equation*}
z\ld x=L^{-1}_z(x)\mathrel{\varrho} L^{-1}_z(y)=z\ld y.
\end{equation*}
In consequence, $\mu(x) \mathrel{\varrho} \mu(y)$, for each $\mu\in \dis{Q}$.

Let $a=\alpha(e), b=\beta(e), c=\gamma(e), d=\delta(e)\in Qe$, with $\alpha,\beta,\gamma,\delta\in\dis Q$ and let $a\mathrel{\varrho} b$ and $c\mathrel{\varrho} d$. Thus, $\gamma(a)\mathrel{\varrho} \gamma(b)$ and $\beta(c)\mathrel{\varrho} \beta(d)$. Hence,
\begin{equation*}
c+a=\gamma(e)+\alpha(e)=\gamma\alpha(e)=\gamma(a)\mathrel{\varrho} \gamma(b)=\gamma\beta(e)=\gamma(e)+\beta(e)=c+b,
\end{equation*}
and similarly $(b+c)\mathrel{ \varrho} (b+d)$
which implies
\begin{equation*}
(a+c)\mathrel{\varrho} (c+b) \mathrel{\varrho} (b+d).
\end{equation*}
Therefore, $\varrho_e:=\varrho|_{Qe}$ (the restriction of $\varrho$ to the orbit $Qe$) is a congruence relation of the abelian group $\orb Qe$.
Furthermore, since $\varrho_e$ is a congruence on the subquandle $Qe$, i.e. $(e\ast a)\mathrel{\varrho_e} (e\ast b)$ and $(e\ld a)\mathrel{\varrho_e} (e\ld b)$, one obtains
$$t^n\cdot a=(1-\varphi_{e,e})^n (a) \mathrel{\varrho_e} (1-\varphi_{e,e})^n (b)=t^n\cdot b$$
and $\varrho_e$ is a congruence of the module $\orb Qe$.
Then, for each $e\in E$, $M_e:=\{x\in Qe\mid x\mathrel{\varrho_{e}} e\}$ is a submodule of the module $\orb Qe$. Hence, for each $x\in M_e$ and $f\in E$
\begin{equation*}
\varphi_{e,f}(x)=(x\ast f-e\ast f) \mathrel{\varrho_f} (e\ast f-e\ast f)=f.
\end{equation*}
So, condition \eqref{cond:lm3:congruences} is satisfied.

Now let $Q$ be the sum of an affine mesh $(A_i;\,\varphi_{i,j};\,c_{i,j})$ over a set $I$ and assume that there is a family $(M_i)_{i\in I}$ of modules such that each $M_i$ is a $\Z[t,t^{-1}]$-submodule of the module $A_i$, for which the condition \eqref{cond:lm3:congruences} is satisfied. Let us consider a relation $\varrho \subseteq \pi$ defined by \eqref{cond:lm2:congruences}.

Clearly, $\varrho$ is an equivalence relation. The reflexivity and symmetry of $\varrho$ are obvious. Now let us note that, since $Q$ is a disjoint union of the sets $A_i$, if for $a,b,c\in Q$, $a\mathrel{\varrho} b$  and $b\mathrel{\varrho} c$ then there is $i\in I$ such that $a,b,c\in A_i$ and $a\equiv_{M_i} b \equiv_{M_i} c$ which shows that the relation  $\varrho$ is transitive. Further, let $a,b,c,d\in Q$ and suppose that $a\mathrel{\varrho} b$ and $c\mathrel{\varrho} d$. Then there are $i,j\in I$, such that $a,b\in A_i$, $a\equiv_{M_i} b$,
$c,d\in A_j$ and $c\equiv_{M_j} d$. By Condition \eqref{cond:lm3:congruences} we have
$\varphi_{i,k}(a)\equiv_{ M_k}\varphi_{i,k}(b)$ and $\varphi_{j,k}(c)\equiv_{M_k}\varphi_{j,k}(d)$, for any $k\in I$.

By the definition of the multiplication in the sum of an affine mesh, we immediately obtain
\begin{equation*}
a\ast c=c_{i,j}+\varphi_{i,j}(a)+(1-\varphi_{j,j})(c)\equiv_{M_j} c_{i,j}+\varphi_{i,j}(b)+(1-\varphi_{j,j})(d)=b\ast d.
\end{equation*}
Obviously, $c-\varphi_{i,j}(a)-c_{i,j}\equiv_{M_j}d-\varphi_{i,j}(b)-c_{i,j}$, and as a consequence we get
\begin{align*}
a\ld c&=(1-\varphi_{j,j})^{-1}(c-\varphi_{i,j}(a)-c_{i,j})=t^{-1}\cdot(c-\varphi_{i,j}(a)-c_{i,j})\equiv_{M_j}t^{-1}\cdot(d-\varphi_{i,j}(b)-c_{i,j})=\\
&=(1-\varphi_{j,j})^{-1}(d-\varphi_{i,j}(b)-c_{i,j})=b\ld d.
\end{align*}
Hence, $\varrho$ is a congruence of $Q$.
\end{proof}

Examples of such constructed congruences are given in the following lemma:

\begin{lm}\label{lmex:clas2}
Let~$Q$ be a medial quandle which is the sum of an affine mesh $(A_i; \varphi_{i,j}; c_{i,j})$ over a set~$I$.
Let~$j\in I$.
Then the following families of submodules satisfy Condition~$(\ref{cond:lm3:congruences})$:
\begin{itemize}
 \item $(\Ker(\varphi_{i,j}))_{i\in I}$;
  \item $(\bigcap_{n\in\mathbb{N}}\varphi_{i,i}^{n}(A_i))_{i\in I}$;
  \item $(\varphi_{i,j}(\bigcap_{n\in\mathbb{N}}\varphi_{i,i}^{n}(A_i)))_{i\in I}$.
\end{itemize}
\end{lm}

\begin{proof}
Note that for any $i,j,p,r\in I$,

    \begin{align}
    &\varphi_{i,j}(\Ker(\varphi_{i,p}))\subseteq \Ker(\varphi_{j,r});\label{eq:1}\\
    &\varphi_{i,j}(\bigcap_{n\in\mathbb{N}}\varphi_{i,i}^{n}(A_i))\subseteq \bigcap_{n\in\mathbb{N}}\varphi_{j,j}^{n}(A_{j});\label{eq:2}\\
    &\varphi_{j,p}\varphi_{i,j}(\bigcap_{n\in\mathbb{N}}\varphi_{i,i}^{n}(A_i))=\varphi_{i,p}(\bigcap_{n\in\mathbb{N}}\varphi_{i,i}^{n}(A_i)). \label{eq:3}
    \end{align}

Indeed. Let $x\in \Ker(\varphi_{i,p})$. This implies that $\varphi_{i,p}(x)=0$ and by (M3), $\varphi_{j,r}\varphi_{i,j}(x)=\varphi_{p,r}\varphi_{i,p}(x)=0$.
Hence $\varphi_{i,j}(x)\in  \Ker(\varphi_{j,r})$.

Now let $x\in \bigcap_{n\in\mathbb{N}}\varphi_{i,i}^{n}(A_i)$. Then for every $n\in \mathbb{N}$ there exists $a_n\in A_i$ such that $x=\varphi_{i,i}^{n}(a_n)$. So, by (M3), for every $n\in \mathbb{N}$
				\begin{align*}
				\varphi_{i,j}(x)=\varphi_{i,j}\varphi_{i,i}^{n}(a_n)=\varphi_{j,j}^{n}\varphi_{i,j}(a_n)\in \varphi_{j,j}^{n}(A_j),
\end{align*}
which implies $\varphi_{i,j}(x)\in \bigcap_{n\in\mathbb{N}}\varphi_{j,j}^{n}(A_{j})$.
Moreover, once again by (M3),

\begin{align*}
\varphi_{j,p}\varphi_{i,j}(\bigcap_{n\in\mathbb{N}}\varphi_{i,i}^{n}(A_i))=\varphi_{i,p}\varphi_{i,i}(\bigcap_{n\in\mathbb{N}}\varphi_{i,i}^{n}(A_i))=
\varphi_{i,p}(\bigcap_{n\in\mathbb{N}}\varphi_{i,i}^{n}(A_i)).
\end{align*}

Hence all three families satisfy Condition~(\ref{cond:lm3:congruences}).
\end{proof}

Lemma \ref{lmexm1:congr} shows that just one submodule of any orbit module is sufficient to define a congruence relation on $Q$ below $\pi$.
\begin{lm}\label{lmexm1:congr}
Let $Q$ be a medial quandle which is the sum of an affine
mesh $(A_i;\,\varphi_{i,j};\,c_{i,j})$ over a set $I$. Let, for some
$i_0\in I$, $M_{i_0}$ be a $\Z[t,t^{-1}]$-submodule
of $A_{i_0}$ and $M_i=\varphi_{i_0,i}(M_{i_0})$, for $i\neq i_0$.
Then the family $(M_i)_{i\in I}$ satisfies Condition~$(\ref{cond:lm3:congruences})$.
\end{lm}

\begin{proof}
Note that $\varphi_{i_0,i_0}(M_{i_0})=(1-t)\cdot M_{i_0}\subseteq M_{i_0}$ and, for $i\neq i_0$, $\varphi_{i,j}(M_i)=\varphi_{i,j}\varphi_{i_0,i}(M_{i_0})=\varphi_{j,j}\varphi_{i_0,j}(M_{i_0})=
 \varphi_{j,j}(M_j)=(1-t)\cdot M_j\subseteq M_j$. Hence, the tuple of submodules $(M_i)_{i\in I}$ satisfies the condition \eqref{cond:lm3:congruences}.
\end{proof}

We use this observation in the sequel to construct congruences with only one non-trivial class.
The idea is the following:

\begin{exm}\label{exm:congr} Let $Q$ be a medial quandle which is the sum of an affine mesh
$(A_i;\,\varphi_{i,j};\,c_{i,j})$ over a set $I$
and assume that there exists $i_0\in I$ such that $\bigcap\limits_{j\in I}\mathrm{Ker}(\varphi_{i_0,j})\neq\{0\}$.
Let
$M_{i_0}\subseteq
\bigcap\limits_{j\in I}\mathrm{Ker}(\varphi_{i_0,j})$ be a non-trivial $\Z[t,t^{-1}]$-submodule
of $A_{i_0}$.

Then, by Lemma~\ref{lmexm1:congr} and Theorem \ref{thm:congruences}, the relation
\begin{equation*}
a\;\alpha\; b\;\; {\rm if \; and \; only \; if}\;\; a=b \;\; {\rm
or} \;\; (a,b\in A_{i_0} \text{ and }a\equiv_{M_{i_0}} b)
\end{equation*}
is a non-trivial congruence of $Q$, such that $\alpha\subseteq
\pi$ and $\alpha|_{A_i}$ is trivial, for each $i_0\neq i\in I$.
\end{exm}

Up to isomorphism there is only one two element SI medial quandle, namely the two element
projection quandle. Moreover, this is also the only 1-reductive SI medial quandle since, by Theorem \ref{thm:decomposition}, each  projection quandle is the sum of one element orbits. In what follows we will study only SI medial quandles which are not projection, i.e. which have at least three elements and, in a non-connected case, have at least one orbit with at least two elements.

Let $a,b\in Q$ and $\Theta(a,b)$ denote the smallest congruence on $Q$ collapsing $(a,b)$. Recall, by \cite[Theorem 1.20.]{MMT}
$\Theta(a,b)$ can be described by the following recursion:
\begin{align*}
X_0&=\{(a,b),(b,a)\}\cup\{(q,q)\mid q\in Q\}\\
X_{n+1}&=X_n\cup \{(x\ast x',y\ast y')\mid (x,y),(x',y')\in X_n\}\cup \{(x\ld x',y\ld y')\mid (x,y),(x',y')\in X_n\}\cup\\
&\cup\{(x,z)\mid (x,y),(y,z)\in X_n\;\;{\rm for\; some\;} y\in Q\}.
\end{align*}
Then $\Theta(a,b)=\bigcup_{n\in\mathbb{N}}X_n$.

Recall also that the least nonzero congruence $\mu$ of $Q$ is called \emph{the monolith} and for any pair $(a,b)\in \mu$ with $a\neq b$,  $\mu=\Theta(a,b)$. In subdirectly irreducible modules, the word used for the smallest proper submodule
is the {\em monolith} too; to distinguish between the congruence and the submodule, we shall call this submodule
the {\em socle} of the module. Recall, that, for general modules, the socle is the sum of all simple submodules. SI modules are sometimes called \emph{cocyclic}, \emph{monolithic} or \emph{completely uniform}.

\begin{thm}\label{thm:simodule}
Let $Q$ be a non-projection SI medial quandle which is the sum of an affine mesh
$(A_i;\,\varphi_{i,j};\,c_{i,j})$ over a set $I$.
Then there is $i_0\in I$ such that $A_{i_0}$ is a subdirectly irreducible $\Z[t,t^{-1}]$-module and the socle of the module $A_{i_0}$ is a non-trivial block of the monolith of $Q$.
\end{thm}

\begin{proof}
By Theorem \ref{thm:congruences}, for each $i\in I$, there is an order preserving mapping $\alpha_i$ from the interval $[0;\pi]$ in the congruence lattice of $Q$ restricted to the orbit $A_i$, to the lattice of $\Z[t,t^{-1}]$-submodules of $A_i$. By Lemma \ref{lmexm1:congr}, any non-trivial submodule $N_i$ od $A_i$ defines a non-trivial congruence $\nu\subseteq\pi$ on the quandle $Q$ such that $N_i=\alpha_i(\nu)$.

Let $\mu$ be the monolith of $Q$. Since $\mu$ is included in any non-trivial congruence of $Q$, we have that $\alpha_i(\mu)\subseteq\alpha_i(\nu)=N_i$ for any non-trivial submodule $N_i$ of $A_i$. This implies that $\alpha_i(\mu)=\{0\}$ or $\alpha_i(\mu)$ is a socle of $A_i$.

By the minimality of the monolith $\mu\subseteq\pi$ there exist $i_0\in I$
and $a \neq b\in A_{i_0}$, such that $\mu=\Theta(a,b)$.
By Theorem \ref{thm:congruences} there are submodules $M_i$ of $A_i$ such that for all $i,j\in I$, $\varphi_{i,j}(M_i)\subseteq M_j$ and for $x,y\in Q$, $(x,y)\in\mu$ if and only if there exists $i\in I$ such that $x,y\in A_i$ and $x-y\in {M_i}$. In particular there exists
a submodule $M_{i_0}=\alpha_{i_0}(\mu)$ of $A_{i_0}$ with $a-b\in M_{i_0}$. This shows that $M_{i_0}$ is the unique non-trivial minimal submodule of the module $A_{i_0}$ and $A_{i_0}$ is a subdirectly irreducible $\Z[t,t^{-1}]$-module.
\end{proof}

\begin{remark}\label{R:locfin}
It is not difficult to observe that subdirectly irreducible modules over finitely generated commutative rings are locally finite. Indeed, each finitely generated SI module over a finitely generated commutative ring has a composition series \cite[Corollary 3.6]{Ja} and each simple module over a finitely generated commutative ring is finite (a field which is finitely generated as a ring is finite). In particular, the socle in a subdirectly irreducible $\Z[t,t^{-1}]$-module must be finite.
\end{remark}

\section{Medial quandles of quasi-affine type}\label{sec:si}
In this section we will prove that all SI medial quandles of quasi-affine type are latin, thus connected. We will also show that non-connected finite SI medial quandles are reductive.

Recall, connected medial quandles are Alexander ones. Let $m\geq 1$ be a natural number and let $Q$ be an Alexander quandle $(A,f)$. Then
$$(((x\ast\underbrace{y)\ast y)\ast \dots )\ast y}_{m-\text{times}}=(1-f)^m(x)+(1-(1-f)^m)(y),$$
and $Q$ is $m$-reductive if and only if $(1-f)^m=0$. In
particular, the orbit $(A_i,1-\varphi_{i,i})$ of a medial quandle is $m$-reductive if
and only if $\varphi_{i,i}^m=0$.
Similarly,
$$(((x\ld\underbrace{ y)\ld y)\ld \dots )\ld y}_{m-\text{times}}=(1-f^{-1})^m(x)+(1-(1-f^{-1})^m)(y).$$
Moreover,
$$
(1-f)^m=0 \mathrel{\Leftrightarrow} ((1-f)f^{-1})^m=0 \mathrel{\Leftrightarrow} (f^{-1}-1)^m=0 \mathrel{\Leftrightarrow} (1-f^{-1})^m=0,
$$
which means that an Alexander quandle is $m$-reductive with respect to multiplication if and only if it is $m$-reductive with respect to left division.
This observation can be generalized for all medial quandles, not only Alexander ones;
it is actually a consequence of the following theorems, presented in \cite{JPSZ}.

\begin{theorem}\cite[Theorem 6.6]{JPSZ}\label{thm:reductive}
Let $Q$ be a medial quandle. Then the following statements are equivalent.
\begin{enumerate}
    \item $Q$ is reductive.
    \item At least one orbit of $Q$ is reductive.
    \item All the orbits of $Q$ are reductive.
\end{enumerate}
\end{theorem}

\begin{prop}\cite[Proposition 6.2]{JPSZ}\label{prop:reductive}
Let $\mathcal A=(A_i;\,\varphi_{i,j};\,c_{i,j})$ be an indecomposable affine mesh over a set $I$. Then the sum of $\mathcal A$ is $m$-reductive if and only if, for every $i\in I$, $$\varphi_{i,i}^{m-1}=0.$$
\end{prop}

In particular, a medial quandle $Q$ is $m$-reductive if and only if all the orbits of $Q$ are $(m-1)$-reductive and
is 2-reductive if and only if every orbit is a projection quandle.
\vskip 3mm

\begin{rem}\label{rem:qa}
Recall that using Kearnes' classification of SI modes \cite{K99} one can conclude that all SI medial quandles are
either of quasi-affine type or of set type.
Let $Q$ be a non-projection SI medial quandle with the monolith $\mu$ and let $A_{i_0}$ be a subdirectly irreducible $\Z[t,t^{-1}]$-module, for some $i_0\in I$, which exists by Theorem \ref{thm:simodule}.
If $S_{i_0}\subseteq A_{i_0}$ is the socle of $A_{i_0}$, then either $\varphi_{i_0,i_0}|_{S_{i_0}}$ is a bijection or $\varphi_{i_0,i_0}|_{S_{i_0}}=0$. In the first case, $S_{i_0}$ is a latin quandle and $Q$ is of quasi-affine type. In the second case, $S_{i_0}$ is a projection quandle and $Q$ is of set type. Since, as shown in \cite{K99}, algebras supported by non-trivial monolith blocks embed into one another, one obtains that either all of them are latin quandles, or projection ones.
\end{rem}

For latin quandles we conclude that each block of their non-trivial congruences, as a subalgebra, has a cancellative binary operation.
It was shown in \cite[Proposition 7.2]{HSV} that a finite medial quandle is connected if and only if it is latin.
Hence, finite connected SI medial quandles are of quasi-affine type.
Moreover, each infinite subdirectly irreducible idempotent medial quasigroup is an example of infinite connected SI medial quandle of quasi-affine type.
An example is the following:

\begin{example}\label{exm:Prugr1}
Let $\Z_{p^\infty}$ be the Pr\"ufer group where $p$ is a prime. According to \cite[Theorem 3.29]{Ber} Pr\"ufer groups are the only infinite SI abelian groups. There are many representations of such groups, e.g.
$$\Z_{p^\infty}=\left\{ \left[ \frac a{p^k}\right]_\sim \mid a,k\in\mathbb{N}\right\}$$
where $\frac{a}{p^k}\sim\frac{b}{p^n}$ if and only if $ap^n\equiv bp^k \pmod{p^{k+n}}$,
and the operation~$+$ on $\Z_{p^\infty}$ is the same as in~$\mathbb{Q}$.

Take now $Q$ the Alexander quandle $(\Z_{p^\infty},-1)$ for $p>2$ a prime. The multiplication by~$2$ is a bijection on $\Z_{p^\infty}$ and therefore~$Q$ is connected and latin.
\end{example}

Example \ref{exm:Prugr} shows that there exist also infinite connected SI medial quandles which are of set type.
On the other hand, each medial quandle of quasi-affine type is connected.

\begin{theorem}\label{thm:KK}
Let $Q$ be a subdirectly irreducible medial quandle of quasi-affine type. Then $Q$ is latin.
\end{theorem}

\begin{proof}
Let $Q$ be an SI medial quandle of quasi-affine type and let $Q$ be the sum of an affine mesh
$(A_i;\,\varphi_{i,j};\,c_{i,j})$ over a set $I$. By Theorem \ref{thm:simodule} there is $i_0\in I$ such that $A_{i_0}$ is a subdirectly irreducible $\Z[t,t^{-1}]$-module and, by Remark \ref{rem:qa},
$\varphi_{i_0,i_0}|_{S_{i_0}}$ is a bijection. By (M3), for each $i\in I$, $\varphi_{i_0,i}|_{S_{i_0}}$ and $\varphi_{i,i_0}|_{\varphi_{i_0,i}(S_{i_0})}$ must be injective.  This implies $S_{i_0}\cap \Ker(\varphi_{i_0,i})=\{0\}$ and $\varphi_{i_0,i}(S_{i_0})\cap \Ker(\varphi_{i,i_0})=\{0\}$.
If for some $i\in I$, $\Ker(\varphi_{i_0,i})\neq \{0\}$ or $\Ker(\varphi_{i,i_0})\neq \{0\}$, then submodules $S_{i_0}$ and respectively, $\Ker(\varphi_{i_0,i})$ or $\Ker(\varphi_{i,i_0})$, according to Lemma~\ref{lmexm1:congr},
define two non-trivial congruences  $\varrho_1$ and $\varrho_2$ on $Q$ such that $\varrho_1\cap\varrho_2$ is a trivial congruence on $Q$.

Hence, we can assume that, for each $i\in I$, $\Ker(\varphi_{i_0,i})=\{0\}$ and $\Ker(\varphi_{i,i_0})=\{0\}$. This implies, again by (M3), that for each $i\in I$, $\Ker(\varphi_{i,i}^2)=\{0\}$ and in consequence, for every $i,j\in I$, $\Ker(\varphi_{i,j})=\{0\}$.

Let now $M_i$ be a finitely generated
submodule of~$A_i$, for some $i\in I$. The module $M_i$ is subdirectly irreducible since it embeds into $A_{i_0}$ through the mapping~$\varphi_{i,i_0}$.
By Remark \ref{R:locfin}, $M_i$ is finite.
Since $\varphi_{i,i}=1-t$, we see that $\varphi_{i,i}$ embeds $M_i$ into~$M_i$. Therefore $\varphi_{i,i}$ is
a bijection on~$M_i$.

We now prove that $\varphi_{i,i}$ is a surjection on~$A_i$.
Indeed, let $a\in A_i$. The cyclic module generated by~$a$ is finite hence $\varphi_{i,i}$ is a bijection of~$\langle a\rangle$
and therefore there exists $b\in \langle a\rangle$ such that $\varphi_{i,i}(b)=a$.
This means that each $\varphi_{i,i}$ is an automorphism of~$A_i$ and
therefore all orbits are latin. Then, by \cite[Proposition 5.2]{JPSZ}, all orbits are isomorphic as quandles and by \cite[Theorem 5.5]{JPSZ}
$Q$ is isomorphic to a direct product of a latin quandle and a projection quandle. Since $Q$ is an SI medial quandle of quasi-affine type then $Q$ must be connected with one latin orbit.
\end{proof}

Note that Theorem \ref{thm:KK} is in fact a consequence of a deep result from universal algebra and as the referee suggests its proof follows quite immediately
from the proof of~\cite[Theorem 3.6]{K99}:
let $Q$ be an SI medial quandle of quasi-affine type with the monolith $\mu$. Thus each algebra $A$ supported by the monolith block is abelian, in the sense that in the direct power $A^2$, the diagonal $\{(a,a)\mid a\in A\}$ is a block of a congruence on $A^2$ \cite[Theorem 7.30]{Ber}. Hence, each such block has a Mal'cev operation $m(x,y,z)$. By \cite[Lemma 3.5 and Theorem 3.6]{K99} we can lift the identities $m(x,y,y)=x=m(y,y,x)$ from the monolith blocks to the algebra $Q$. Therefore, $m(x,y,z)$ is a Mal'cev term for $Q$ which implies that $Q$ is polynomially equivalent to an idempotent reduct of a module \cite[Corollary 6.3.2]{RS}. In particular, $Q$ is a cancellative Alexander quandle $(Q,f)$ in which $m(x,y,z)=x-y+z$. On the other side $m(x,y,z)$ is a derived operation of $(Q,f)$. Therefore $1-f$ is surjective and $Q$ is latin.

\begin{cor}\label{cor:canorb}
Let $Q$ be an SI medial quandle which is the sum of an affine mesh
$(A_i;\,\varphi_{i,j};\,c_{i,j})$ over a set $I$ and such that for each $i\in I$, $|A_i|> 1$. If $Q$ has at least one cancellative orbit then $Q$ is connected.
\end{cor}
\begin{proof}
Let $Q$ be an SI medial quandle with orbits which have at least two elements and suppose that for some $i\in I$, $\Ker(\varphi_{i,i})=\{0\}$. Then by (M3), $\Ker(\varphi_{i,j})=\{0\}$ for each $j\in I$. By Theorem \ref{thm:simodule} there exists $i_0\in I$ such that $A_{i_0}$ is a subdirectly irreducible module. Since $\varphi_{i,i_0}$ is an injection and, by assumption $|A_i|>1$, we obtain that $\varphi_{i,i_0}(A_i)$ is a cancellative subquandle of the Alexander quandle $(A_{i_0},1-\varphi_{i_0,i_0})$. Hence, the socle $S_{i_0}\subseteq \varphi_{i,i_0}(A_i)$ is also cancellative and it follows that $Q$ is of quasi-affine type. By Theorem \ref{thm:KK}, $Q$ must be connected.
\end{proof}

By Theorem \ref{thm:KK} all non-connected SI medial quandles are of set type. The next result shows that in a finite case such quandles are reductive.
\begin{theorem}\label{finite_si}
Let $Q$ be a finite non-connected non-projection subdirectly irreducible medial quandle.
Then $Q$ is reductive.
\end{theorem}
\begin{proof}
Let $Q$ be a sum of an indecomposable affine mesh $(A_i;\varphi_{i,j};c_{i,j})$ over a set $I=\{1,\ldots ,n\}$ for $n\geq 2$. By Theorem \ref{thm:simodule} there exists $i_0\in I$ such that $A_{i_0}$ is a subdirectly irreducible module with the socle $S_{i_0}$. Since $Q$ is of set type, $\varphi_{i_0,i_0}\mid_{S_{i_0}}=0$ and hence $\Ker\varphi_{i_0,i_0}\neq\{0\}$.
Since $Q$ is finite, there exists $k$ such that $\varphi_{i_0,i_0}^{k+1}(A_{i_0})=\varphi_{i_0,i_0}^{k}(A_{i_0})$. Consider two submodules of the module $A_{i_0}$: $\Ker\varphi_{i_0,i_0}$ and $\varphi_{i_0,i_0}^k(A_{i_0})$. Since their intersection is trivial and $A_{i_0}$ is SI, one obtains that $\varphi_{i_0,i_0}^k(A_{i_0})=\{0\}$ and $A_{i_0}$ is reductive. By Theorem \ref{thm:reductive} this means that $Q$ is reductive.
\end{proof}

Theorem \ref{finite_si} cannot be extended to infinite quandles because there exist infinite SI medial quandles that are
neither reductive nor connected, see e.g. Example~\ref{exm:Prugr}. On the other hand, an SI medial quandle
cannot be both connected and reductive
or, in a finite case, non-connected and Alexander.

\begin{remark}\label{prop:non-connected}
The only medial quandle which is both connected and reductive consists of exactly one element. Indeed, let $Q$ be a connected non-trivial medial quandle. Then the mapping $\varphi_{1,1}$ in its canonical mesh has to be surjective, and therefore $\varphi_{1,1}^n=0$ for no number~$n$. Moreover, all reductive SI medial quandles are of set type.
\end{remark}

\section{Reductivity versus nilpotency}\label{sec:reductivity}
Theorem \ref{finite_si} shows that finite reductive medial quandles play an important role in the theory of medial quandles. Reductivity of medial quandles is closely related to the notion of nilpotency. A medial quandle $Q$ is
reductive if and only if its left multiplication group $\lmlt{Q}$ is nilpotent, as we shall see later in Theorem~\ref{th:nilpotent}.
Further, a nilpotent $\varphi_{i,i}$ can appear in an indecomposable affine mesh $\mathcal A=(A_i;\,\varphi_{i,j};\,c_{i,j})$
if and only if the sum of $\mathcal A$ is reductive.

In the sequel we focus on the nilpotency of $\lmlt{Q}$. We start with two auxiliary lemmas.

\begin{lemma}\label{lem:conjugation}\cite[Lemma 2.4]{JPZ17}
 Let~$Q$ be a medial quandle, $e$, $f\in Q$ and $\alpha\in\dis Q$.
 Then $^{L_{e}}\alpha=\; ^{L_{f}}\alpha$.
\end{lemma}

By $G'$ we will denote the \emph{commutator subgroup} (or \emph{derived subgroup}) of a group $G$, i.e. the subgroup generated by all the commutators of the group $G$.

It was proved in~\cite[Section 2]{HSV} that for any quandle~$Q$, $\lmlt{Q}'\trianglelefteq\dis Q$. If~$Q$ is medial, it means that $\lmlt{Q}'$ is abelian and therefore $[\lmlt{Q}',\lmlt{Q}']=1$. Although the following lemma is formulated for general groups, its setting in our
context results in a strong relation to nilpotency.

\begin{lemma}\label{lem:commutators}
 Let $G$ be a group with $[G',G']=1$. Then, for each $\alpha,\beta\in G'$ and $a,b,c,d\in G$:
 \begin{enumerate}
  \item $[a,bc]=[a,b]\cdot\; ^b[a,c]$,
  \item $[\alpha\beta,c]=[\alpha,c]\cdot[\beta,c]$,
  \item $[\;^b\alpha,c]=\;^{^cb}[\alpha,c]$,
  \item $[[a,bc],d]=[[a,b],d]\cdot\; ^{^db}[[a,c],d]$.
 \end{enumerate}
\end{lemma}

\begin{proof}
(1) $[a,bc]=abca^{-1}c^{-1}b^{-1}=aba^{-1}[a,c]b^{-1}=[a,b]b[a,c]b^{-1}=[a,b]\cdot\;^b[a,c]$\\
 (2) $[\alpha\beta,c]=\alpha\beta c\beta^{-1}\alpha^{-1}c^{-1}=\alpha[\beta,c]c\alpha^{-1}c^{-1}
 =\alpha c\alpha^{-1}c^{-1}[\beta,c]=[\alpha,c]\cdot[\beta,c]$\\
 (3) $[\;^b\alpha,c]=b\alpha b^{-1}cb\alpha^{-1}b^{-1}c^{-1}=
 b\alpha [b^{-1},c]c\alpha^{-1}b^{-1}c^{-1}=
 b[b^{-1},c]\alpha c\alpha^{-1}b^{-1}c^{-1}=\;
 ^cb[\alpha,c]\;(^cb)^{-1}
 =\;^{^cb}[\alpha,c]$\\
 (4) $[[a,bc],d]=[[a,b]\cdot\;^b[a,c],d]=[[a,b],d]\cdot [\;^b[a,c],d]=[[a,b],d]\cdot\; ^{^db}[[a,c],d]$
\end{proof}

\begin{theorem}\label{th:nilpotent}
 Let $Q$ be a medial quandle and let $m\geq 1$. Then $Q$ is strictly $m$-reductive if and only if $\lmlt{Q}$ is
 nilpotent of degree $m-1$.
\end{theorem}

\begin{proof}
A quandle is a projection quandle
if and only if its left multiplication group is trivial, so the theorem holds for $m=1$.
\vskip 2mm
Let $m>1$ and let $\mathcal A=(A_i;\,\varphi_{i,j};\,c_{i,j})$,
 for $i,j\in I$, be the canonical affine mesh for~$Q$.
 Choosing an element~$e_i\in A_i$, every element of~$A_i$ can be written as $\alpha(e_i)$,
 for some $\alpha\in\lmlt{Q}$. Moreover,
 according to \cite[Lemma~3.8]{JPSZ}, we have $\varphi_{i,j}(\alpha(e_i))=[\alpha,L_{e_i}](e_j)$, for any $\alpha\in\lmlt{Q}$
 and $i,j\in I$.

 According to Proposition~\ref{prop:reductive}, $Q$ is $m$-reductive if and only if $\varphi_{i,i}^{m-1}=0$, for each~$i\in I$.
 This condition can be equivalently rewritten as $[\ldots[[\alpha,L_{e_i}],L_{e_i}],\ldots,L_{e_i}](e_i)=e_i$, for each~$i\in I$
 and $\alpha\in\lmlt{Q}$,
 which means $[\ldots[[\alpha,L_{e_i}],L_{e_i}],\ldots,L_{e_i}]=1$.

 ``$\Leftarrow$'' If $\lmlt{Q}$ is nilpotent of degree~$m-1$ then
 $[\ldots[\alpha_1,\alpha_2],\ldots,\alpha_m]=1$, for any $\alpha_j\in\lmlt{Q}$,
 in particular for $\alpha_2=\cdots=\alpha_m=L_{e_i}$, for any $i\in I$.

 ``$\Rightarrow$'' According to Lemma~\ref{lem:conjugation}, we have $[[\alpha,L_{e_i}],L_{e_i}]=[[\alpha,L_{e_i}],L_{e_j}]$,
 for any $i,j$.
 Hence $m$-reductivity implies $[\ldots[[\alpha,L_{e_{i_1}}],L_{e_{i_2}}],\ldots,L_{e_{i_{m-1}}}]=1$,
 for any $e_{i_k}$, $1\leq k\leq m-1$. Now we should inductively enlarge this property
 to $[\ldots[[\alpha,\beta_1],\beta_2],\ldots,\beta_{m-1}]$, for any $\beta_1,\ldots,\beta_{m-1}\in\lmlt{Q}$.
 Suppose $\beta_j=\hat\beta_j\bar\beta_j$. According to Lemma~\ref{lem:commutators} (4),
 \begin{multline*}
 [\ldots[\ldots[[\alpha,\beta_1],\beta_2],\ldots,\hat\beta_j\bar\beta_j],\ldots,\beta_{m-1}]=\\
 [\ldots[\ldots[[\alpha,\beta_1],\beta_2],\ldots,\hat\beta_j],\ldots,\beta_{m-1}]
 \cdot
% ^{\begin{picture}(52.1,0.1)
%    \put(0,0){${}^{{}^{{}^{{}^{\beta_{m-1}}\ddots}\beta_{j+1}}\hat\beta_j}$} 
%   \end{picture}}
 \strut^{\gamma_{j}}
 [\ldots[\ldots[[\alpha,\beta_1],\beta_2],\ldots,\bar\beta_j],\ldots,\beta_{m-1}],
 \end{multline*}
 where $\gamma_{m}=1$ and $\gamma_k=\strut^{\gamma_{k+1}}\beta_k$,
 for each $k=j+1,\ldots,m-1$, and $\gamma_j=\strut^{\gamma_{j+1}}\hat\beta_j$.
 The right side of the equation is trivial, due to the induction hypothesis. Hence $\lmlt{Q}$ is a nilpotent group of degree at most~$m-1$.
\end{proof}

By Proposition \ref{prop:reductive} we know that the orbits of a strictly $m$-reductive medial quandle are $(m-1)$-reductive. But not necessarily
strictly $(m-1)$-reductive -- the degree of reductivity may be even smaller.
It was nevertheless proved in~\cite[Theorem 6.6]{JPSZ} that the orbits cannot be $(m-4)$-reductive in such a case.
We improve the result here showing that no $(m-3)$-reductive orbits can appear.

\begin{lm}\label{lem:m-reductive}
Let $m\geq 1$ and $\mathcal A=(A_i;\,\varphi_{i,j};\,c_{i,j})$ be an indecomposable affine mesh over a set $I$. Assume there is $j\in I$ such that $\varphi_{j,j}^m=0$. Then $\varphi_{i,i}^{m+1}=0$ for every $i\in I$.
\end{lm}

\begin{proof}
First note that applying (M3) $m$-times, for any $k\in I$, we have
\begin{equation}\label{eq:lem:m-reducive}
\varphi_{k,k}^m\varphi_{j,k}=\varphi_{j,k}\varphi_{j,j}^m=0,
\end{equation}
because $\varphi_{j,j}^m=0$ by assumption.

The indecomposability condition says that
the group $A_k$ is generated by all the elements $c_{i,k}$, $\varphi_{i,k}(a)$ with $i\in I$ and $a\in A_i$.
So it is sufficient to verify that
$\varphi_{k,k}^{m+1}\varphi_{i,k}=0$ and
$\varphi_{k,k}^{m+1}(c_{i,k})=0$, for every $i\in I$.

By \eqref{eq:lem:m-reducive} and (M4) we have
\begin{equation*}
\varphi_{k,k}^{m+1}(c_{i,k}-c_{j,k})=\varphi_{k,k}^m\varphi_{j,k}(c_{i,j})=0.
\end{equation*}
This implies
\begin{equation*}
\varphi_{k,k}^{m+1}(c_{i,k})=\varphi_{k,k}^{m+1}(c_{j,k}),
\end{equation*}
for all $i,k\in I$. In particular, for $i=k$, we see that
$\varphi_{k,k}^{m+1}(c_{k,k})=0$, and thus
$\varphi_{k,k}^{m+1}(c_{j,k})=0$. This gives
$\varphi_{k,k}^{m+1}(c_{i,k})=0$ for every $i\in I$.

Further, by (M3) applying $(m+1)$-times,
\begin{equation*}
\varphi_{k,k}^{m+1}\varphi_{i,k}=\varphi_{j,k}\varphi_{j,j}^{m}\varphi_{i,j}=0,
\end{equation*}
for every $i\in I$. Hence, $\varphi_{k,k}^{m+1}=0$, for every $k\in I$.
\end{proof}

\begin{corollary}\label{cor1:reductive}
Let $Q$ be a medial quandle.
If one orbit of $Q$ is $m$-reductive, then $Q$ is $(m+2)$-reductive.
\end{corollary}

We recall now a few results about 2-reductive medial quandles.

\begin{lm}\cite[Lemma 6.8]{JPSZ}\label{lem:2-reductive}
Let $\mathcal A=(A_i;\,\varphi_{i,j};\,c_{i,j})$ be an indecomposable affine mesh over a set $I$. Assume there are $j,k\in I$ such that $\varphi_{j,k}=0$. Then $\varphi_{i,k}=0$ for every $i\in I$.
\end{lm}

\begin{theorem}\cite[Theorem 6.9]{JPSZ}\label{thm:2-reductive}
Let $Q$ be a medial quandle and assume it is the sum of an indecomposable affine mesh $(A_i;\,\varphi_{i,j};\,c_{i,j})$ over a set $I$. Then the following statements are equivalent.
\begin{enumerate}
    \item $Q$ is 2-reductive.
    \item For every $j\in I$, there is $i\in I$ such that $\varphi_{i,j}=0$.
    \item $\varphi_{i,j}=0$ for every $i,j\in I$.
\end{enumerate}
\end{theorem}

In particular, medial quandles with a one-element orbit are always
2-reductive and with a two-element orbit are 3-reductive.

We know by now that a strictly $m$-reductive medial quandle has $(m-1)$-reductive orbits
and may have $(m-2)$-reductive orbits too. Some of the orbits might even be isomorphic.
But none of the mappings~$\varphi_{i,j}$ is a bijection.
\vskip 3mm

\begin{prop}\label{lm:bijection}
Let $Q$ be a reductive medial quandle which is the sum of an
indecomposable affine mesh $(A_i;\,\varphi_{i,j};\,c_{i,j})$ over a
set $I$. If, for $i,k\in I$, $\left|A_i\right|>1$ or $\left|A_k\right|>1$, then the homomorphism $\varphi_{i,k}$
is not a bijection.
\end{prop}

\begin{proof}
Since medial quandles with a
one-element orbit are always 2-reductive and by Theorem
\ref{thm:2-reductive}, in 2-reductive medial quandles
$\varphi_{l,j}=0$, for every $l,j\in I$, we may assume that
$|A_i|>1$.

Let us suppose that there are $i, k\in I$ such that $\varphi_{i,k}\colon
A_i\to A_k$ is a bijection.
By (M3)
$\varphi_{k,k}\varphi_{i,k}=\varphi_{i,k}\varphi_{i,i}$.
It implies
\begin{equation}\label{eq:lm:bijection}
\varphi_{i,k}^{-1}\varphi_{k,k}=\varphi_{i,i}\varphi_{i,k}^{-1}.
\end{equation}
On the other hand, by (M4) we have
\begin{equation*}
c_{j,i}=\varphi_{i,k}^{-1}\varphi_{k,k}(c_{j,k}-c_{i,k})=\varphi_{i,i}\varphi_{i,k}^{-1}(c_{j,k}-c_{i,k})\in \im(\varphi_{i,i}),
\end{equation*}
for any $j\in I$.
By indecomposability, it means that $A_i=\langle \im(\varphi_{j,i})\mid j\in I\rangle$.

Since $Q$ is a reductive medial quandle, there is a natural number
$m>0$ such that $\varphi_{k,k}^m=0$.
Let $m$ be the least such a number.
If $m=1$, then by Lemma
\ref{lem:2-reductive}, $\varphi_{j,k}=0$, for every $j\in I$. In
particular, $\varphi_{i,k}=0$.

For $m>1$, once again by (M3) one has
\begin{equation}\label{eq2:lm:bijection}
\varphi_{k,k}^{m-1}\varphi_{i,k}\varphi_{j,i}=\varphi_{k,k}^{m-1}\varphi_{k,k}\varphi_{j,k}=\varphi_{k,k}^m\varphi_{j,k}=0,
\end{equation}
for any $j\in I$. Since $\varphi_{i,k}$ is a bijection
and
$A_i=\langle \im(\varphi_{j,i})\mid j\in I\rangle$,
Condition \eqref{eq2:lm:bijection} implies that $\varphi_{k,k}^{m-1}=0$,
a contradiction with the minimality of~$m$.
Hence there are no $i,k\in I$ such that
$\varphi_{i,k}$ is a bijection.
\end{proof}

At the end of this section we shall prove that reductive medial quandles have a property
which is crucial for our considerations and
enables us to find small congruences that are good candidates for monoliths.

\begin{lm}\label{lem2:m-reductive}
Let $m\geq 2$ and $Q$ be an indecomposable affine mesh $\mathcal
A=(A_i;\,\varphi_{i,j};\,c_{i,j})$ over a set $I$ and assume that
for some $i,j\in I$, $i\neq j$,
$\varphi_{i,i}^{m-1}=\varphi_{j,j}^{m-1}=0$.
Then $\varphi_{i,j}\varphi_{i,i}^{m-2}=\varphi_{j,j}^{m-2}\varphi_{i,j}=0$.
\end{lm}

\begin{proof}
For $m=2$ the conclusion follows by Lemma \ref{lem:2-reductive}. Let $m>2$.
Clearly, by (M3),
\begin{equation}\label{eq1:lem:m-reducive}
\varphi_{i,j}\varphi_{i,i}^{m-2}=\varphi_{i,j}\varphi_{i,i}^{m-3}\varphi_{i,i}=
\varphi_{j,j}\varphi_{j,j}^{m-3}\varphi_{i,j}=\varphi_{j,j}^{m-2}\varphi_{i,j}.
\end{equation}

By indecomposability, it is sufficient to verify that
$\varphi_{i,j}\varphi_{i,i}^{m-2}\varphi_{k,i}=0$ and
$\varphi_{i,j}\varphi_{i,i}^{m-2}(c_{k,i})=0$, for every $k\in I$.

For each $k\in I$,  by (M3) we have
\begin{equation*}
\varphi_{i,j}\varphi_{i,i}^{m-2}\varphi_{k,i}=\varphi_{i,j}\varphi_{i,i}\varphi_{i,i}^{m-3}\varphi_{k,i}=
\varphi_{j,j}\varphi_{j,j}\varphi_{j,j}^{m-3}\varphi_{k,j}=\varphi_{j,j}^{m-1}\varphi_{k,j}=0
\end{equation*}
and by \eqref{eq1:lem:m-reducive} and (M4)
\begin{equation*}
\varphi_{i,j}\varphi_{i,i}^{m-2}(c_{k,i})=\varphi_{j,j}^{m-2}\varphi_{i,j}(c_{k,i})=\varphi_{j,j}^{m-2}\varphi_{j,j}(c_{k,j}-c_{i,j})=
\varphi_{j,j}^{m-1}(c_{k,j}-c_{i,j})=0.\qedhere
\end{equation*}
\end{proof}

\begin{proposition}\label{prop:theta}
Let~$Q$ be a non-projection reductive quandle that is the sum of an affine mesh $(A_i; \varphi_{i,j}; c_{i,j})$ over a set~$I$.
Then there exists $i\in I$ such that $\bigcap_{j\in I} \Ker (\varphi_{i,j})\neq\{0\}$.
\end{proposition}

\begin{proof}
Assume that $Q$ is strictly $m$-reductive, for some $m\geq 2$.
If $m=2$ then by Theorem \ref{thm:2-reductive}, $\varphi_{i,j}=0$, for all~$i,j\in I$.

Suppose now $m> 2$. Since $Q$ is strictly $m$-reductive then by
Proposition \ref{prop:reductive},
there is at least one orbit of $Q$, say $A_1$, which is strictly
$(m-1)$-reductive.

Hence, there is an element $0\neq a_1\in A_1$
such that the elements:
$0,a_1,\varphi_{1,1}(a_1),\ldots,\varphi_{1,1}^{m-2}(a_1)\in A_1$
are pairwise different.
By Lemma \ref{lem2:m-reductive},
$0\neq\varphi_{1,1}^{m-2}(a_1)\in\bigcap_{j\in I}\mathrm{Ker}(\varphi_{1,j})$.
\end{proof}

As we have seen in the proof, for a quandle $Q$ which is strictly $m$-reductive, for some $m\geq 2$, every strictly $(m-1)$-reductive orbit has the property described in Proposition \ref{prop:theta}. Hence
if $\bigcap_{j\in I} \Ker (\varphi_{i,j})=\{0\}$, for some $i\in I$, then the orbit
has to be $(m-2)$-reductive.

\section{Medial quandles of set type}\label{sec:set-type}

We have already shown in Section~4 that all SI medial quandles of quasi-affine type are latin and finite non-connected SI medial quandles are reductive.
But this dichotomy between \emph{latin} and \emph{reductive} holds only in the finite case, there are infinite SI medial quandles that are
neither latin nor reductive.

\begin{example}\label{exm:Prugr}
Take $Q$ the Alexander quandle $(\Z_{p^\infty},1-p)$ where $p$ is a prime.
The multiplication by~$p$ is surjective on $\Z_{p^\infty}$ and therefore~$Q$ is connected. It follows by Remark~\ref{prop:non-connected} that $Q$ is non-reductive.
Since the multiplication by~$p$ is not injective, $Q$ is not latin.
The kernel is $\{\left[ \frac a{p}\right]_\sim\mid a\in\Z_p\}$, i.e. the socle of $\Z_{p^\infty}$ which
is the minimal $\Z$-submodule of~$\Z_{p^\infty}$ and corresponds to the monolith of
the subdirectly irreducible quandle~$Q$.
\end{example}

The above example is clearly of set type -- the socle is isomorphic to the Alexander quandle $(\Z_p,1)$.
It is hence useful to find a more general condition that is satisfied by all reductive medial quandles
as well as by all SI medial quandles of set type;
 this turns out to be the property described in Proposition \ref{prop:theta}.

\begin{de}
 Let~$Q$ be a medial quandle being the sum of an indecomposable affine mesh $(A_i;\,\varphi_{i,j};\,c_{i,j})$ over a set $I$.
 The quandle is called {\em quasi-reductive} if there exists $i\in I$ such that $\bigcap_{j\in I}\Ker(\varphi_{i,j})$ is non-trivial.
\end{de}

An alternative definition, not using the notion of an affine mesh and hence applicable to all quandles, is the following:

 \begin{de}
 Let~$Q$ be a quandle. We say that $Q$ is {\em quasi-reductive} if
 there exist $a,b\in Q$, different elements lying in the same orbit, such that $L_a=L_b$.
\end{de}

The equivalence of the definitions is clear: on one hand, if $a\in\bigcap_{j\in I}\Ker(\varphi_{i,j})$
then $a*c=\varphi_{i,j}(a)+\varphi_{j,j}(c)+c_{i,j}=0+\varphi_{j,j}(c)+c_{i,j}=0*c$, for all $c\in Q$.
On the other hand we analogously obtain $\varphi_{i,j}(a-b)=0$, for all $j\in I$.

Proposition \ref{prop:theta} tells us that every non-projection reductive medial quandle
is quasi-reductive. The converse implication is not true, in Example~\ref{exm:Prugr} we see
medial quandles that are quasi-reductive but not reductive. There even exist finite quasi-reductive
medial quandles that are not reductive, for instance the Alexander quandles
$\left( \Z_{p}^3,\  \text{\strut\tiny$\begin{pmatrix}2 & 0 & 0\\ 0 & 1 & 1 \\ 0 & 0 & 1\end{pmatrix}$}\right)$ where $p>2$ is a prime.

\begin{mthm}\label{all_si}
 Every subdirectly irreducible medial quandle with more than two elements is either latin or quasi-reductive.
\end{mthm}

\begin{proof}
Let~$Q$ be an SI medial quandle. If~$Q$ is a projection quandle, then every equivalence is a congruence and hence
$|Q|=2$.
So, we can assume that $Q$ is a non-projection SI medial quandle.

If $Q$ is connected, then it is an Alexander quandle $(A,f)$ with $(1-f)(A)=A$. If $\Ker(1-f)=\{0\}$, then $Q$ is latin.
Otherwise, $Q$ is infinite and quasi-reductive.

Let $Q$ be a sum of an indecomposable affine mesh $(A_i;\varphi_{i,j};c_{i,j})$ over a set $I$ with $|I|\geq 2$.
Since $Q$ is non-connected, by Corollary \ref{cor:canorb}, either there is $k\in I$ with $|A_k|=1$ or for each $i\in I$, $\Ker(\varphi_{i,i})\neq\{0\}$. By Theorem \ref{thm:2-reductive}, in the former case, $Q$ is strictly 2-reductive, and thus obviously quasi-reductive.

Suppose now that for each $i\in I$, $\Ker(\varphi_{i,i})\neq\{0\}$ and there are $i\neq s\in I$ such that $\Ker(\varphi_{i,s})=\{0\}$.
Then by (M3), for each $t\in I$ one obtains $\varphi_{s,t}\varphi_{i,s}=\varphi_{i,t}\varphi_{i,i}$,
which implies $\Ker(\varphi_{s,t})\neq\{0\}$.

Therefore, if for each $i\in I$, $\Ker(\varphi_{i,i})\neq\{0\}$, then there is $j\in I$ such that for every $k\in I$, $\Ker(\varphi_{j,k})\neq\{0\}$.
If $\bigcap_{k\in I}\Ker(\varphi_{j,k})\neq\{0\}$, then of course, $Q$ is quasi-reductive. Assume that $\bigcap_{k\in I}\Ker(\varphi_{j,k})=\{0\}$. Then there are $0\neq a\in A_j$ and $k_1\neq k_2\in I$ such that $\varphi_{j,k_1}(a)=0$ and
$\varphi_{j,k_2}(a)\neq0$. This implies that for every $k\in I$ we have
\begin{align*}
0=\varphi_{k_1,k}\varphi_{j,k_1}(a)=\varphi_{k_2,k}\varphi_{j,k_2}(a).
\end{align*}
Hence, for every $k\in I$, $0\neq\varphi_{j,k_2}(a)\in \Ker(\varphi_{k_2,k})$. So we obtain that $\bigcap_{k\in I}\Ker(\varphi_{k_2,k})\neq\{0\}$ which means that $Q$ is quasi-reductive.
\end{proof}

\begin{corollary}\label{cor:ncAlex}
The only subdirectly irreducible non-connected Alexander quandle has two elements.
\end{corollary}

\begin{proof} Assume for contradiction $(A,f)$ is a non-connected SI Alexander quandle with more than two elements. Then $(A,f)$ is non-projection and $\im(1-f)\neq\{0\}$.
By Theorem \ref{all_si} the Alexander quandle $(A,f)$ is quasi-reductive, which implies that $\Ker(1-f)\neq\{0\}$.
Each orbit of an Alexander quandle $(A,f)$ is isomorphic to the quandle $(\im(1-f),f)$ (\cite[Proposition 7.1]{HSV}). Let $\Ker(1-f)\cap \im(1-f)=\{0\}$. This means that each orbit $(\im(1-f),f)$ is latin and therefore the quandle $(A,f)$ is isomorphic to a direct product of a latin quandle and a projection quandle (\cite[Theorem 5.5]{JPSZ}). Since $(A,f)$ is SI we must have $\Ker(1-f)\cap \im(1-f)\neq\{0\}$.
Let for $a,b\in A$
$$(a,b)\in\rho_1 \Leftrightarrow\begin{cases}
                                 a=b,\;{\rm or}\\
                                 a\in \im(1-f) \wedge b\in \im(1-f) \wedge a-b\in\Ker(1-f),
                                \end{cases}$$
and let
$$(a,b)\in\rho_2 \Leftrightarrow\begin{cases}
                                 a=b, \;{\rm or}\\
                                 a\notin \im(1-f) \wedge b\notin \im(1-f) \wedge a-b\in\Ker(1-f).
                                \end{cases}$$
We claim that both $\rho_1$ and $\rho_2$ are congruences. Let $a\equiv_{\rho_i} \bar a$
and $b\equiv_{\rho_i} \bar b$ for $i\in\{1,2\}$:
\begin{align*}
 a*b &= (1-f) (a) + f(b) = (1-f) (a-b) + b = (1-f)(\bar a-\bar b) + b \equiv_{\rho_i} (1-f)(\bar a-\bar b) + \bar b=\bar a* \bar b.
 \end{align*}
Since $\Ker(1-f)=\Ker(1-f^{-1})$ as well as $\im(1-f)=\im(1-f^{-1})$ and $a\backslash b = (1-f^{-1})(a-b) + b$ we immediately obtain that also
$a\backslash b\equiv_{\rho_i} \bar a\backslash \bar b$.

Now $\rho_1$ and $\rho_2$ are congruences that intersect trivially, a contradiction to the existence of such a quandle.
\end{proof}

By Theorem \ref{thm:congruences}, congruences of a connected Alexander quandle $Q=(A,f)$ and congruences of the
$\Z[t,t^{-1}]$-module~$A$ coincide. In particular, by Theorem \ref{thm:simodule}, the quandle $Q$ is subdirectly irreducible if
and only if $A$ is a subdirectly irreducible $\Z[t,t^{-1}]$-module. Thus classifying SI Alexander quandles is a task equivalent to classifying SI medial quasigroups which is equivalent to classifying SI $\Z[t,t^{-1}]$-modules.

Note the equivalence of SI Alexander quandles and SI modules holds only under the condition of connectedness.
Indeed, subdirectly irreducible modules can give non-connected quandles \zie{which} are not SI, according to Corollary~\ref{cor:ncAlex} and Example \ref{Ex:6.6.} below.

\begin{example}\label{Ex:6.6.}
Consider the Alexander quandle $(\Z_4,3)$. It is isomorphic to the sum of the affine mesh $$((\Z_{2},\Z_{2});\,\left(\begin{smallmatrix}0&0\\0&0\end{smallmatrix}\right);\,
\left(\begin{smallmatrix}0&1\\1&0\end{smallmatrix}\right)).$$
It is clear that there are two minimal elements in the lattice of its congruences: $\{\{0,2\},\{1\},\{3\}\}$
and $\{\{1,3\},\{0\},\{2\}\}$, so $(\Z_4,3)$ is not subdirectly irreducible. However, $\Z_4=Z_{2^2}$ is a subdirectly irreducible $\Z$-module.
\end{example}

\section{Subdirectly irreducible non-connected quasi-reductive quandles}\label{sec:subdirectly irreducible finite reductive quandles}

The structure of connected medial quandles depends on the structure of the underlying module.
According to Main Theorem~\ref{all_si}, every non-connected SI medial quandle is quasi-reductive.
Hence, what remains and what is
the aim of this section, is to describe all non-connected subdirectly irreducible quasi-reductive quandles. We start with the construction of a congruence that will play the crucial role
in our considerations.

Let $Q$ be a quandle and let $\lambda$ be the relation on $Q$ defined by
\begin{equation*}
a\;\lambda \;b\;\ {\rm iff}\;\; \forall (x\in Q)\; a\ast x=b\ast x \;\ {\rm iff}\;\; \forall (x\in Q)\; a\ld x=b\ld x.
\end{equation*}
Obviously, $\lambda$ is a congruence on $Q$~(see~\cite[Section 8.6]{RS}) and each of its blocks is a projection quandle.
Set
\begin{equation}\label{eq:theta}
\theta:=\pi\cap\lambda.
\end{equation}

\begin{lm} Let $Q$ be a medial quandle which is the sum of an indecomposable affine mesh $(A_i;\,\varphi_{i,j};\,c_{i,j})$ over a set $I$. Then $M_{\theta|_{A_i}}=\bigcap\limits_{j\in I}\mathrm{Ker}(\varphi_{i,j})$, for each $i\in I$.
\end{lm}
\begin{proof}
To see it, let $a,b\in A_i$. Then
\begin{align*}
a\;\theta\; b\;\;\Leftrightarrow\;\;\forall (j\in I)\;\forall (x\in A_j)\; a\ast x=b\ast x\;\;\Leftrightarrow\\
\forall (j\in I)\;\forall (x\in A_j) \;\;c_{i,j}+\varphi_{i,j}(a)+(1-\varphi_{j,j})(x)=c_{i,j}+\varphi_{i,j}(b)+(1-\varphi_{j,j})(x)\;\;\Leftrightarrow\\
\forall (j\in I) \;\;\varphi_{i,j}(a)=\varphi_{i,j}(b)\;\;\Leftrightarrow\;\; \forall (j\in I) \;\;a-b\in \mathrm{Ker}(\varphi_{i,j})\;\;\Leftrightarrow
a-b\in \bigcap\limits_{j\in I}\mathrm{Ker}(\varphi_{i,j}).
\end{align*}
\end{proof}

An alternative definition of quasi-reductivity says that a medial quandle is quasi-reductive if and only if~$\theta$
is non-trivial on~$Q$. The class of quasi-reductive medial quandles contains all non-projection reductive medial
quandles, according to Proposition~\ref{prop:theta}.
But not every quasi-reductive quandle is reductive, examples are given by the non-connected Alexander quandle $(\Z_6,-1)$ or connected Alexander quandles constructed in Example \ref{exm:Prugr}. See also Examples \ref{exm:pp} and \ref{exm:pp1}.

From now on, we shall suppose that $Q$ is quasi-reductive and non-connected SI medial quandle which is the sum of an
indecomposable affine mesh $(A_i;\,\varphi_{i,j};\,c_{i,j})$ over (at least two element)
set $I$. By Theorem \ref{thm:KK}, the quandle $Q$ is of set type. Let $\mu$ be the monolith congruence of $Q$. By Theorem \ref{thm:simodule} there is $i_0\in I$ such that $A_{i_0}$ is a subdirectly irreducible module and
the socle $S$ of  $A_{i_0}$ is a non-trivial block of $\mu$. In the sequel, we set $i_0=1$.
By Example \ref{exm:congr} and the minimality of the monolith $\mu$, $A_1$ is the only orbit where $\mu$ and $\theta$ are non-trivial, i.e. $S\subseteq\bigcap_{j\in I}\Ker(\varphi_{1,j})\neq\{0\}$.
Example \ref{exm:pp1} shows that the relations $\mu$ and $\theta$ do not have to be equal.

Note that for each $a\in Q$, $L_a|_{Qe}$ is an automorphism of the orbit $Qe$ for any element $e\in Q$.

\begin{lm}\label{lem4:subdir} Let $a,b\in Q\setminus A_1$. If for each
$x\in A_1$, $a\ast x=b\ast x$, then $a=b$.
\end{lm}

\begin{proof}
Assume for contradiction $a\neq b\in Q\setminus A_1$ and $a\ast x=b\ast x$, for every $x\in A_1$. Then
$\Theta(a,b)|_{A_1}=\Delta$.
Since the relation $\theta|_{(Q\setminus A_1)}$ is also trivial, it
follows that $\Theta(a,b)\cap\theta=\Delta$. So $Q$ can
not be subdirectly irreducible.
\end{proof}

\begin{cor}\label{cor1:subdir}
For each $1\neq i\in I$, the homomorphism $\varphi_{i,1}\colon A_i\to A_1$ is an injection.
\end{cor}

\begin{proof}
Let $1\neq i\in I$, and $\varphi_{i,1}(a)=\varphi_{i,1}(b)$ for some $a,b\in A_i$.
Hence, for any $x\in A_1$
\begin{align*}
a\ast
x=c_{i,1}+\varphi_{i,1}(a)+(1-\varphi_{1,1})(x)=c_{i,1}+\varphi_{i,1}(b)+(1-\varphi_{1,1})(x)=b\ast
x.
\end{align*}
Hence, by Lemma \ref{lem4:subdir}, $a=b$.
\end{proof}

By Corollary \ref{cor1:subdir}, we can
assume that, for each $1\neq i\in I$, the orbit $A_i$ is a submodule of
$A_1$. In such case, $A_i\subseteq A_1$, and $\varphi_{i,1}=1_{A_i}$.
Furthermore, by (M3) it follows that for each $i,j\in I\setminus\{1\}$ and $a\in A_i\subseteq A_1$
\begin{equation*}
\varphi_{i,i}(a)=\varphi_{i,1}\varphi_{i,i}(a)=\varphi_{1,1}\varphi_{i,1}(a)=\varphi_{1,1}(a),
\end{equation*}
hence $\varphi_{i,i}=\varphi_{1,1}|_{A_i}$. Moreover,
\begin{align*}
\varphi_{1,i}=\varphi_{i,1}\varphi_{1,i}=\varphi_{1,1}^2,\\
\varphi_{i,j}=\varphi_{j,1}\varphi_{i,j}=\varphi_{1,1}\varphi_{i,1}=\varphi_{1,1}|_{A_i}.
\end{align*}
Each summand~$A_i$ can be structurally viewed either as an $R$-module (where $R$ is a suitable homomorphic image of the ring $\Z[t,t^{-1}]$) or as a permutation group acting
on~$Q$. We need both the features and therefore the summands will be treated either as modules or
as permutation groups, according to our needs.

Let, for each $j\in I$, denote by $0_j$  the neutral element of~$A_j$.
\begin{lm}\label{cor2:subdir}
For each $i\neq j\in I\setminus\{1\}$, $c_{i,1}\neq c_{j,1}$.
\end{lm}

\begin{proof}
Suppose, that there are $i\neq j\in I\setminus\{1\}$ such that
$c_{i,1}=c_{j,1}$. Then for all $x\in A_1$, $0_i\ast x=0_j\ast x$.
Hence, by Lemma \ref{lem4:subdir}, $0_i=0_j$, a contradiction.
\end{proof}

\begin{lm}\label{lem6:subdir}
For each $ i\neq j\in I\setminus\{1\}$, the constants $c_{i,j}$ are uniquely determined only by the constants $c_{i,1}\in A_1$.
\end{lm}

\begin{proof}
It straightforwardly follows by (M4) that, for any $i,j,k\in I$, $\varphi_{j,k}(c_{i,j})=\varphi_{k,k}(c_{i,k}-c_{j,k})$.
Hence, for $k=1$ and $j\neq 1$ we obtain
\begin{align*}
c_{i,j}=\varphi_{j,1}(c_{i,j})=\varphi_{1,1}(c_{i,1}-c_{j,1}).
\end{align*}
In particular, for $j\neq 1$ and $i=1$
\begin{align*}
c_{1,j}=\varphi_{1,1}(c_{1,1}-c_{j,1})=-\varphi_{1,1}(c_{j,1}).
\end{align*}
\end{proof}

Let $\varphi:=\varphi_{1,1}$. Corollary \ref{cor1:subdir} and Lemma \ref{lem6:subdir} directly imply the following lemma.
\begin{lm}\label{lem7:subdir}
For each $1\neq i\in I$, $A_i=\varphi(A_1)$.
\end{lm}
\begin{proof}
Let $1\neq j\in I$. By indecomposability, for each $1\neq i\in I$,
the group $A_i$ is generated by the following sets:
$\varphi_{1,i}(A_1)=\varphi^2(A_1)$,
$\varphi_{j,i}(A_j)=\varphi(A_j)$, and all elements
$c_{1,i}=-\varphi(c_{i,1})$, and $c_{j,i}=\varphi(c_{j,1}-c_{i,1})$.
Since each $A_j$ is a subgroup of $A_1$, it is evident that
$A_i\subseteq \varphi(A_1)$.

On the other hand, the group $A_1$ is generated by the following sets:
$\varphi_{1,1}(A_1)=\varphi(A_1)$, $\varphi_{j,1}(A_j)=A_j$, and all
constants $c_{j,1}$. Hence, $\varphi(A_1)$ is generated by
$\varphi^2(A_1)$, $\varphi(A_j)$, and $\varphi(c_{j,1})$, which shows
that $A_i=\varphi(A_1)$ for each $1\neq i\in I$.
\end{proof}

\begin{lm}\label{lem8:subdir}
For $i\neq j$, $c_{i,1}-c_{j,1}\notin \varphi(A_1)$.
\end{lm}
\begin{proof}
Assume
$c_{i,1}-c_{j,1}\in\varphi(A_1)=A_j$,  for some $i\neq j\in I$.
Then there exists $a\in A_j$
such that $c_{i,1}=c_{j,1}+a$. But, for each $b\in A_1$,
$0_i\ast b=(1-\varphi)(b)+c_{i,1}=(1-\varphi)(b)+a+c_{j,1}=a\ast b$.
Then, by Lemma~\ref{lem4:subdir}, $0_i=a$, a contradiction with Lemma \ref{cor2:subdir}.
\end{proof}

Lemma \ref{lem8:subdir}
gives an upper bound for the
number of orbits in a non-connected SI quasi-reductive
medial quandle.

\begin{cor}\label{thm2:subdir} Let $\kappa=\left | A_1/\varphi(A_1)\right|$. The number of orbits in $Q$
is at most $\kappa+1$.
\end{cor}

\begin{rem}\label{rem:twoorb}
According to~Corollary~\ref{thm2:subdir} a non-connected infinite quasi-reductive SI medial quandle must have at least one infinite orbit.
Moreover, if the group homomorphism $\varphi$ is onto then $Q$  has exactly two orbits.
\end{rem}

Now we are ready to describe the structure of any non-connected SI quasi-reductive medial quandle.
Actually, the already known structure of connected ones can be formulated in the very same theorem.
The theorem is, on purpose, not formulated in the language of affine meshes although they can be clearly visible in it.

\begin{thm}\label{thm1:subdir}
Let $A$ be a subdirectly irreducible $\Z[t,t^{-1}]$-module. Suppose that
$\varphi:a\mapsto (1-t)\cdot a$ is a non-injective endomorphism of~$A$.
Let $C$ be a subset of a transversal to $\varphi(A)$ in~$A$
such that $C\cup\varphi(A)$ generates~$A$.
We denote by $\siq (A,t,C)$ the set $A\cup(\varphi(A)\times C)$
equipped with the following operation~$*$:
\begin{align*}
 a*b&=\varphi(a)+t\cdot b, && \text{for } a,b\in A\\
 (a,i)*(b,j)&=(\varphi(a+i-j)+t\cdot b,j), && \text{for }a,b\in\varphi(A),\ i,j\in C\\
 (a,i)*b&=     a+t\cdot b+i,&&\text{for }a\in\varphi(A),\ i\in C,\ b\in A\\
 a * (b,j) &= (\varphi(\varphi(a)-j)+t\cdot b,j),&&\text{for }a\in A,\ b\in\varphi(A),\ j\in C\\
\end{align*}
and with the operation~$\backslash$ defined as the left division with respect to~$*$.

Then the algebra $\siq(A,t,C)$ is a subdirectly irreducible quasi-reductive medial quandle.
Conversely, every subdirectly irreducible quasi-reductive medial quandle
is isomorphic to $\siq (A,t,C)$, for some $A$ and~$C$.
\end{thm}

\begin{rem}\label{rem:transl}
We translate the construction of $\siq(A,t,C)$ into the language of meshes. To be consistent with the rest of the
section, we suppose that $1$ is a formal symbol not belonging to~$A$ and we put $I=C\cup \{1\}$.
Now
\begin{itemize}
\item $A_1=A$, $A_i=\varphi(A)\times\{i\}$, for each $i\in C$,
\item $\varphi_{1,1}=\varphi_{i,j}=\varphi$, $\varphi_{i,1}=1$ (i.e. the identity mapping), and $\varphi_{1,j}=\varphi^2$ for each $i,j\in C$,
\item $c_{i,1}=i$, $c_{1,i}=-\varphi(i)$, $c_{i,j}=\varphi(i-j)$ for each $i,j\in C$.
\end{itemize}
\end{rem}
\begin{proof}[Proof of Theorem \ref{thm1:subdir}]
It is easy to check that $\mathcal{A}=(A_i;\varphi_{i,j};c_{i,j})$ given in Remark \ref{rem:transl} is an indecomposable affine mesh
and that $\siq(A,t,C)$ is equal to the sum of~$\mathcal{A}$ over~$I$.

``$\Leftarrow$'' Let $Q$ be a SI quasi-reductive medial
quandle.
If~$Q$ is connected then $Q$ is an Alexander quandle $(A,t)$,
for some subdirectly irreducible $\Z[t,t^{-1}]$-module~$A$.
But then $Q=\siq(A,t,\emptyset)$.

If~$Q$ is not connected then
the proof that $Q$ has the form described in the theorem
follows from Theorem \ref{thm:simodule}, Lemmas \ref{cor2:subdir}, \ref{lem6:subdir}, \ref{lem7:subdir},
\ref{lem8:subdir}, Corollaries \ref{cor1:subdir},
and \ref{thm2:subdir}.

``$\Rightarrow$'' Now let $Q$ be the sum of the affine mesh described above.
Then, by assumption, $A_1=A$ is a subdirectly irreducible $\Z[t,t^{-1}]$-module.
Let $M$ be the socle of the module~$A_1$.
Since $\Ker(\varphi)$ is a non-zero submodule of~$A$, clearly $M\subseteq\Ker(\varphi)$.
By Example \ref{exm:congr}, the relation $\Upsilon\subseteq Q\times Q$
defined as follows:
\begin{equation*}
a\;\Upsilon\; b\;\; {\rm if \; and \; only \; if}\;\; a=b \;\; {\rm
or} \;\; (a,b\in A_1 \text{ and }a\equiv_M b)
\end{equation*}
is a congruence of the quandle $Q$.

To prove that $Q$ is subdirectly irreducible we will show that for any $a\neq b\in Q$ the congruence $\Theta(a,b)$ generated by
$a$ and $b$ contains the congruence $\Upsilon$.

It is obvious for $a,b\in A_1$. Now we will show that for $a$ or $b$ in $Q\setminus A_1$, the congruence $\Theta(a,b)|_{A_1}$ is non-trivial.
We will divide the proof into several cases.

Case 1. Let $a,b\in A_i=\varphi(A)$ for $1\neq i\in I$. It is easy to notice that for any $x\in A_1$
\begin{align*}
a\ast x=c_{i,1}+\varphi_{i,1}(a)+(1-\varphi_{1,1})(x)=c_{i,1}+a+(1-\varphi_{1,1})(x)\neq\\
c_{i,1}+b+(1-\varphi_{1,1})(x)=c_{i,1}+\varphi_{i,1}(b)+(1-\varphi_{1,1})(x)=b\ast x.
\end{align*}

Case 2. Let $a\in A_i=\varphi(A)$ and $b\in A_j=\varphi(A)$ for $i\neq j\in I\setminus\{1\}$.
Then there are $a_1,b_1\in A_1$ such that $a=\varphi(a_1)$ and $b=\varphi(b_1)$.
Furthermore, by the assumption, the constants $c_{i,1}$ and $c_{j,1}$ belong to different cosets of $\varphi(A)$, and hence we have that $c_{i,1}\notin c_{j,1}+\varphi(A)$. This implies that $c_{i,1}\neq c_{j,1}+\varphi(b_1)-\varphi(a_1)$.
Hence for any $x\in A_1$
\begin{align*}
a\ast x=c_{i,1}+\varphi_{i,1}(a)+(1-\varphi_{1,1})(x)=c_{i,1}+a+(1-\varphi_{1,1})(x)=c_{i,1}+\varphi(a_1)+(1-\varphi_{1,1})(x)\neq\\
c_{j,1}+\varphi(b_1)+(1-\varphi_{1,1})(x)=c_{j,1}+b+(1-\varphi_{1,1})(x)=c_{j,1}+\varphi_{j,1}(b)+(1-\varphi_{1,1})(x)=b\ast x.
\end{align*}

Case 3. Let $a\in A_1$, $b\in A_i=\varphi(A)$ for $1\neq i\in I$ and $c_{i,1}\notin\varphi(A)$.
Then there is $b_1\in A_1$ such that $b=\varphi(b_1)$ and $c_{i,1}\neq\varphi(a)-\varphi(b_1)$. In consequence, for any $x\in A_1$, we obtain
\begin{align*}
a\ast x=\varphi_{1,1}(a)+(1-\varphi_{1,1})(x)=\varphi(a)+(1-\varphi_{1,1})(x)\neq \\
c_{i,1}+\varphi(b_1)+(1-\varphi_{1,1})(x)=c_{i,1}+b+(1-\varphi_{1,1})(x)=\\
c_{i,1}+ \varphi_{i,1}(b)+(1-\varphi_{1,1})(x)=b\ast x.
\end{align*}

Case 4. Let $a\in A_1$, $b\in A_i=\varphi(A)$ for $1\neq i\in I$ and $c_{i,1}\in\varphi(A)$. Since, by assumption, the group $A$ is generated by the set $\varphi(A)\cup\{c_{i,1}\mid i\in I\}$, there is $j\in I$, such that $1\neq j\neq i$, with $c_{j,1}\notin\varphi(A)$. Then  $c_{j,1}\neq\varphi^2(a)-\varphi(c_{i,1}-c_{j,1})-\varphi(b)$. Hence
\begin{align*}
a\ast 0_j=c_{1,j}+\varphi_{1,j}(a)+(1-\varphi_{j,j})(0_j)=c_{1,j}+\varphi_{1,j}(a)=-c_{j,1}+\varphi^2(a)\neq\\
\varphi(c_{i,1}-c_{j,1})+\varphi(b)=c_{i,j}+\varphi_{i,j}(b)=c_{i,j}+\varphi_{i,j}(b)+(1-\varphi_{j,j})(0_j)=b\ast 0_j.
\end{align*}
Since $a\ast 0_j, b\ast 0_j\in A_j$, by Case 1 we have that for any $x\in A_1$, $(a\ast 0_j)\ast x\neq (b\ast 0_j)\ast x$.

Hence, by Cases 1--4, for any $a\neq b$ with $a$ or $b$ in $Q\setminus A_1$, $\Theta(a,b)|_{A_1}$ is indeed non-trivial, which shows that $Q$ is subdirectly irreducible.
\end{proof}

Summarizing, in the reductive case, we have that an SI strictly $m$-reductive medial quandle
has exactly one strictly $(m-1)$-reductive orbit and all other orbits  are strictly
$(m-2)$-reductive. As we have written in Section \ref{sec:si}, an Alexander quandle $(A,f)$ is $(m-1)$-reductive if and only if $(1-f)^{m-1}=0$ and $(m-1)$-reductive Alexander quandles are reducts
of modules over the ring $\Z[t,t^{-1}]/(1-t)^{m-1}$. Hence, in the case of $m$-reductive medial quandles, in Theorem \ref{thm1:subdir}, one can replace a SI $\Z[t,t^{-1}]$-module by its homomorphic image $\Z[t,t^{-1}]/(1-t)^{m-1}$.

\begin{corollary}\label{cor:SI_red}
 Let~$Q$ be a subdirectly irreducible $m$-reductive medial quandle. Then
 \begin{itemize}
  \item $|Q|=2$, if $m=1$;
  \item $Q$ is isomorphic to $\siq(A,t,C)$ where $A$ is a subdirectly irreducible $\Z[t,t^{-1}]/(1-t)^{m-1}$-module,
  if $m\geq 2$.
 \end{itemize}

\end{corollary}

In particular, each non-connected finite reductive medial quandle is a sum of an affine-mesh of the following form:
$$((A,\underbrace{\varphi(A),\dots,\varphi(A)}_{n\;\leq |A/\varphi(A)|});
\left(\begin{smallmatrix}\varphi&\varphi^2&\varphi^2&\ldots &\varphi^2& \\
1&\varphi&\varphi&\ldots &\varphi&\\
1&\varphi&\varphi&\ldots &\varphi&\\
\vdots&\vdots& \vdots& \vdots &\vdots &\\
1&\varphi&\varphi&\ldots &\varphi&
\end{smallmatrix}\right);
\left(\begin{smallmatrix}
0      &-\varphi(c_{2,1})       &\ldots & -\varphi(c_{j,1})       &\ldots &-\varphi(c_{n,1})& \\
c_{2,1}&0                       &\ldots & \varphi(c_{2,1}-c_{j,1})               &\ldots &\varphi(c_{2,1}-c_{n,1})& \\
\vdots &\vdots                  & \vdots& \vdots                  & \vdots&\vdots           & \\
c_{i,1}&\varphi(c_{i,1}-c_{2,1})&\ldots & \varphi(c_{i,1}-c_{j,1}) &\ldots &\varphi(c_{i,1}-c_{n,1})& \\
\vdots &\vdots                  & \vdots& \vdots                  &\vdots & \vdots          & \\
c_{n,1}&\varphi(c_{n,1}-c_{2,1})&\ldots &\varphi(c_{n,1}-c_{j,1}) &\ldots &0                &
\end{smallmatrix}\right)).$$

\begin{exm}\label{exm:sired}
Let $m\geq 1$, $p$ be a prime and
$C$ be a coset to $p\Z_{p^m}$ in $\Z_{p^m}$ containing at least one generator of~$\Z_{p^m}$.
Then $\siq(\Z_{p^m},1-p,C)$ is a finite SI $m$-reductive medial quandle with  $|C|+1$ orbits.
\end{exm}

But there are also infinite reductive SI medial quandles.

\begin{example}\label{exm:si3red}
Let $m\in \mathbb{N}$ and $A=\Z_{p^\infty}^m$ where $p$ is a prime. The block $\left[ \frac 0{p}\right]_\sim$ will be denoted by $0$.  Let  $t\cdot (a_1,a_2,\ldots,a_m)=(a_1,a_2-a_1,\ldots,a_m-a_{m-1})$. Then $A$ is a subdirectly irreducible $\Z[t,t^{-1}]$-module
 since the socle
is $\{(0,0,\ldots,0,\left[ \frac c{p}\right]_\sim)\mid \text{ for }c\in\Z_p\}$.

 Let now $\varphi$ be the multiplication by~$1-t$, i.e. $\varphi((a_1,a_2,\ldots,a_m))=(0,a_1,a_2,\ldots,a_{m-1})$. Clearly $\varphi^m=0$. If we take $C=\{(\left[\frac{1}{p^i}\right]_\sim,0,\ldots,0)\mid i\in\mathbb{N}\}$ then
 $\siq(A,t,C)$
 is an $m$-reductive medial quandle with infinitely many orbits.
\end{example}

\begin{exm}\label{exm:pp}
In a Pr\"ufer group $\Z_{p^\infty}$, for some prime $p$, the multiplication by~$p$ is not injective but
$p\Z_{p^\infty}=\Z_{p^\infty}$. Thus we obtain a family $\siq(\Z_{p^\infty},1-p,\{\left[ \frac 0{p}\right]_\sim\})$ of infinite non-connected
SI medial quandles with two isomorphic orbits.
\end{exm}

The last example is more complex. Not only the quandle is not reductive,
we also have $\Ker\varphi\not\subseteq \mathop{\im}\varphi$.

\begin{exm}\label{exm:pp1}
Let $A=\Z_{p^2}\times\Z_{p}[x]$ where $p$ is a prime and let  $t\cdot (a,b_0+b_1x+\ldots+b_mx^m)=(a-pb_0,b_0-b_1+(b_1-b_2)x+\ldots+(b_{m-1}-b_m)x^{m-1}+b_mx^m)$. Then $A$ is a subdirectly irreducible $\Z[t,t^{-1}]$-module with the socle $p\Z_{p^2}\times\{0\}$.

Let now $\varphi$ be the multiplication by~$1-t$, i.e.
\begin{align*}
\varphi((a,b_0+b_1x+\ldots+b_mx^m))=(pb_0,b_1+b_2x+\ldots+b_mx^{m-1}).
\end{align*}
Then $\Ker(\varphi)=\Z_{p^2}\times\{0\}$ and $\varphi(A)=p\Z_{p^2}\times\Z_{p}[x]=\varphi^2(A)$.
Suppose that $C\subseteq\Z_{p^2}$ is a transversal to $p\Z_{p^2}$ in~$\Z_{p^2}$ containing at least one element coprime to~$p$.
Then $\siq(A,t,C\times\{0\})$
is an infinite SI quasi-reductive, but non-reductive, medial quandle with  $|C|+1$ orbits.
\end{exm}

\section{Isomorphisms of subdirectly irreducible medial quandles}\label{sec:isomorphism}

Two subdirectly irreducible connected medial quandles are isomorphic if and only if the underlying modules
are isomorphic.
It remains to decide which quandles constructed in
Theorem \ref{thm1:subdir} are isomorphic.
Let us start with homologous affine meshes introduced in \cite[Definition 4.1]{JPSZ}.
\begin{de}
We call two affine meshes $\mathcal A=(A_i;\varphi_{i,j};c_{i,j})$
and $\mathcal A'=(A_i';\varphi_{i,j}';c_{i,j}')$, over the same
index set $I$, \emph{homologous}, if there is a bijection $\sigma$ of
the set $I$, group isomorphisms $\psi_i:A_i\to A_{\sigma i}'$, and
constants $d_i\in A_{\sigma i}'$, such that, for every $i,j\in I$,
\begin{enumerate}
    \item[(H1)] $\psi_j\varphi_{i,j}=\varphi_{\sigma i,\sigma j}'\psi_i$, i.e., the following diagram commutes:
$$ \begin{CD}
A_i @>\varphi_{i,j}>> A_j\\ @VV\psi_iV @VV\psi_jV\\
A_{\sigma i}' @>\varphi_{\sigma i,\sigma j}'>> A_{\sigma j}'
\end{CD}$$
    \item[(H2)] $\psi_j(c_{i,j})=c_{\sigma i,\sigma j}'+\varphi_{\sigma i,\sigma j}'(d_i)-\varphi_{\sigma j,\sigma j}'(d_j)$.
\end{enumerate}
\end{de}

\begin{theorem}\cite[Theorem 4.2]{JPSZ}\label{thm:isomorphism}
Let $\mathcal A=(A_i;\varphi_{i,j};c_{i,j})$ and $\mathcal
A'=(A_i';\varphi_{i,j}';c_{i,j}')$ be two indecomposable affine
meshes, over the same index set $I$. Then the sums of $\A$ and $\A'$
are isomorphic quandles if and only if the meshes $\A$ and $\A'$ are
homologous.
\end{theorem}

Let $\mathcal{A}=(A_i;\varphi_{i,j};c_{i,j})$ be an indecomposable affine mesh
described in Remark \ref{rem:transl}. In particular, the orbit $A_1=A$ is an Alexander quandle which originates from the abelian group $(A,+)$ of the given $\Z[t,t^{-1}]$-module. We have the following sequence of lemmas.

\begin{lm}\label{lm11:subdir}
For each $i\in I$, let
$c'_{i,1}\in A$ be such that $c'_{i,1}\in
c_{i,1}+\varphi(A)$. Then the sum of $\mathcal{A}$ is
isomorphic to the sum of the indecomposable affine mesh $\mathcal{A}'=(A_i;\varphi_{i,j};c'_{i,j})$.
\end{lm}
\begin{proof}
Let $\sigma=id$, $\psi_1=id$, $d_1=0$ and for every $1\neq i\in I$,
$\psi_i=id$ and $d_i=c_{i,1}-c'_{i,1}\in\varphi(A)=A_i$.
Hence, condition (H1) is satisfied trivially. Moreover,
\begin{align*}
c_{i,1}=c'_{i,1}+d_i-0=c'_{i,1}
+\varphi_{i,1}(d_i)-\varphi_{1,1}(0),
\end{align*}
which shows that the condition (H2) is also satisfied.
\end{proof}

\begin{lm}\label{lm10:subdir}
Let $A/\varphi(A)$ be a cyclic group and $\kappa=|A/\varphi(A)|>1$.
Then
\begin{itemize}
 \item there is exactly one, up to
isomorphism, SI quasi-reductive medial quandle
with two orbits;
\item there is exactly one SI quasi-reductive medial
quandle with three orbits, such that $c_{3,1}\in \varphi(A)$;
\item if $\kappa<\omega$ then there is exactly one SI quasi-reductive medial quandle with $\kappa+1$ orbits.
\end{itemize}
\end{lm}
\begin{proof}
Since $A$ is generated by the set
$\varphi(A)\cup\{c_{i,1}\mid i\in I\}$, at least for one constant $c\in \{c_{i,1}\mid i\in I\}$ a coset $c+\varphi(A)$ must
be a generator of the group $A/\varphi(A)$. Hence, if $Q$ has only two
orbits, the constant $c=c_{2,1}\notin \varphi(A)$ and $c+\varphi(A)$ is one of the generators of $A/\varphi(A)$.
Therefore, there exists an isomorphism $\psi:A\to A$ such that, for
any other $d\notin \varphi(A)$, where $d+\varphi(A)$ is a generator of $A/\varphi(A)$, we have
$\psi(c)=d$. Hence, for $\sigma=id$, $\psi_1=\psi_2=\psi$ and
$d_1=d_2=0$, the conditions  (H1) and (H2) are satisfied for affine
meshes:
$((A,\varphi(A));\,\left(\begin{smallmatrix}\varphi&\varphi^2\\1&\varphi\end{smallmatrix}\right);
\,\left(\begin{smallmatrix}0&-\varphi(c)\\c&0\end{smallmatrix}\right))$
and
$((A,\varphi(A));\,\left(\begin{smallmatrix}\varphi&\varphi^2\\1&\varphi\end{smallmatrix}\right);
\,\left(\begin{smallmatrix}0&-\varphi(d)\\d&0\end{smallmatrix}\right))$.
In consequence, they are isomorphic.

The same arguments works in the case of three orbits with
$c_{3,1}=0$. So by Lemma \ref{lm11:subdir} all SI quasi-reductive medial quandles with 3 orbits where
$c_{3,1}\in \varphi(A)$ are isomorphic.

Finally, in the case of $\kappa+1$ orbits, the required condition that,
for each $i\neq j\in I\setminus\{1\}$, $c_{i,1}\notin c_{j,1}+\varphi(A)$,
implies (by Lemma \ref{lm11:subdir}) that there is only one way for choosing constants in
$A$. So, the statement is obvious.
\end{proof}

\begin{lm}\label{lm12:subdir}
Suppose that for some $1\neq i\in I$, $c_{i,1}=0$. Then the sum of $\mathcal{A}$ is not isomorphic to
the sum of the indecomposable affine mesh $\mathcal{A}'=(A_i;\varphi_{i,j};c'_{i,j})$ with
$c'_{i,1}\notin\varphi(A)$ for each $1\neq i\in I$.
\end{lm}

\begin{proof}
Let $c_{i,1}=0$ for some $1\neq i\in I$. Since $0\in
\varphi(A)$ and $c'_{i,j}+\varphi(A)\neq
\varphi(A)$, for any isomorphism $\psi:A\to A$, every
$d_1\in A$ and $d_i\in \varphi(A)$,  we have
\begin{align*}
\psi(c_{i,1})=0\neq
c'_{i,j}+\varphi_{i,1}(d_i)-\varphi_{1,1}(d_1)=c'_{i,j}+d_i-\varphi(d_1)\in
c'_{i,j}+\varphi(A).
\end{align*}
This means that the condition (H2) fails for any isomorphism
$\psi:A\to A$.
\end{proof}

By Lemmas \ref{lm11:subdir} and \ref{lm12:subdir} we immediately
obtain

\begin{cor}
Let $c_{i,1}\in\varphi(A)$ for some $1\neq i\in I$. Then the sum
of $\mathcal{A}$ is not isomorphic to the sum of the indecomposable affine mesh
$\mathcal{A}'=(A_i;\varphi_{i,j};c'_{i,j})$ with $c'_{i,1}\notin\varphi(A)$
for each $1\neq i\in I$.
\end{cor}

In the sequel, the group endomorphism $\mathbb{Z}_n\rightarrow \mathbb{Z}_n$ defined by $x\mapsto ax$ for $a\in\mathbb{Z}_n$ will be denoted by $a$.

\begin{exm}
Let $\varphi=\left(\begin{smallmatrix}0&0\\1&0\end{smallmatrix}\right):
\Z_2^2\to \Z_2^2$ and
$c=\left(\begin{smallmatrix}1\\0\end{smallmatrix}\right)$.

It was shown in \cite[Table 3]{JPSZ}, that up to isomorphism, there are
exactly two reductive, but not 2-reductive medial quandles of size
6:
$$((\Z_{2^{2}},2\Z_{2^{2}});\,\left(\begin{smallmatrix}2&0\\1&2\end{smallmatrix}\right);
\,\left(\begin{smallmatrix}0&-2\\1&0\end{smallmatrix}\right))\;\; {\rm and}\;\;
((\Z_2^2, \varphi(\Z_2^2);\,\left(\begin{smallmatrix}\varphi&0\\1&\varphi\end{smallmatrix}\right);
\,\left(\begin{smallmatrix}0&-\varphi(c)\\c&0\end{smallmatrix}\right)).$$
They can be actually written as $\siq(\Z_4,3,\{1\})$ and $\siq(\Z_2^2,\left(\begin{smallmatrix}1 & 0 \\1 & 1\end{smallmatrix}\right),\{c\})$
and by  Theorem \ref{thm1:subdir} both of them are subdirectly irreducible.

Further, there are nine reductive, but not 2-reductive medial
quandles of size 8. Only two of them are subdirectly irreducible:
$\siq(\Z_4,3,\{0,1\})$ and $\siq(\Z_2^2,\left(\begin{smallmatrix}1 & 0 \\1 & 1\end{smallmatrix}\right),\{0,c\})$.
Note that both
are strictly 3-reductive.
\end{exm}

Let us recall that for an abelian group $A$ and its automorphism $t\in Aut(A,+)$,
$A$ may be considered as a $\Z[t,t^{-1}]$-module by setting
\begin{align*}
t\cdot a=t(a),\qquad\text{for all } a\in A.
\end{align*}

Hence, it is clear that each congruence of
a cyclic group $A$ is a congruence of the $\Z[t,t^{-1}]$-module~$A$.
Consequently, a $\Z[t,t^{-1}]$-module~$A$ with the underlying group
cyclic is subdirectly irreducible if and only if $A$ is subdirectly irreducible as an abelian
group. The only cyclic SI groups are
 groups $\Z_{p^s}$ of order $p^s$, for some prime
number~$p$.

\begin{exm}\label{ex:lat}
The Alexander quandle $(\Z_{p^r},k)$ where $r>0$, $p$ is a prime
and $k$ coprime to~$p$ is an SI latin quandle.
\end{exm}

The only non-zero nilpotent endomorphisms of the group $\Z_{p^s}$
are of the form $\varphi=p^ka$, for some $0<k<s$ and $a$ coprime with~$p$, and the Alexander quandles
$(\Z_{p^s},1-p^ka)$ are strictly $(\left\lceil \frac{s}{k} \right\rceil +1)$-reductive.

Clearly, each pair of endomorphisms of a cyclic group conjugates if
and only if they are equal. Hence, by Theorem \ref{thm:isomorphism},
for a group $\Z_{p^s}$ and two different nilpotent endomorphisms of
$\Z_{p^s}$ we always obtain non-isomorphic quandles. So, in non-isomorphic sums of affine meshes with cyclic orbits,
constants must play a crucial role.

Now we give the characterization of non-isomorphic SI finite reductive medial quandles with each orbit cyclic.
\begin{thm}\label{thm2:cyclic}
Let $n\geq 1$, $I=\{1,2,\ldots, n,n+1\}$ and $K=\{2,\ldots, n,n+1\}$. Let $\mathcal{A}=(A_i;\,\varphi_{i,j};\,c_{i,j})$ and $\mathcal{A}'=(A_i;\,\varphi_{i,j};\,c'_{i,j})$ be two
indecomposable affine meshes over $I$ described in Remark \ref{rem:transl} with $A=\Z_{p^s}$, for some prime power $p^s$, and
$\varphi=p^ka$, for some $0<k<s$ and $a$ coprime with~$p$.
Assume that $c_{i,1}, c'_{i,1}\notin\varphi(A)$, for each
$i\in K$, or there is (exactly one) $i\in K$ such that
$c_{i,1}, c'_{i,1}\in\varphi(A)$.
Then the sums $\mathcal{A}$ and $\mathcal{A}'$ are isomorphic if and
only if $n\leq \left\lceil \frac{s}{k} \right\rceil$ or there is a permutation $\sigma$ of the set $K$
such that,
for any $i,j\in K$, the constants
satisfy the following condition:
\begin{align}
c_{i,1}c'_{\sigma(j),1}-c_{j,1}c'_{\sigma(i),1}=0.
\end{align}
\end{thm}

\begin{proof}
By Theorem \ref{thm:isomorphism}, two indecomposable affine meshes
over the same index set $I$ are isomorphic if and only if the meshes
are homologous. Hence, to show that
the meshes $\mathcal{A}$ and $\mathcal{A}'$ are isomorphic it is enough to check
the condition (H2) only for constants $c_{i,1}\in A$, $i\in I$,
(the condition (H1) is trivially satisfied). So, we have to check whether
there are a permutation $\sigma$ of the set $K$,
a group isomorphism $\psi:A\to A$ and constants $d_1\in A$ and
$d_i\in A_{\sigma(i)}$ such that for every $i\in K$,
\begin{align*}
\psi(c_{i,1})=c'_{\sigma(i),1}+\varphi_{\sigma(i),1}(d_i)-\varphi_{1,1}(d_1)=c'_{\sigma(i),1}+d_i-\varphi(d_1).
\end{align*}
Since for $i\neq
1$, $d_i\in \varphi(A)$, then there is $a_i\in A$ such that
$d_i=\varphi(a_i)$ and $d_i-\varphi(d_1)=\varphi(a_i-d_1)\in \varphi(A)$. Therefore our problem can be reformulated in the following way: Are
there a permutation $\sigma$ of the set $K$, a group
isomorphism $\psi:A\to A$ and constants $r_i\in \varphi(A)$ such
that for every $i\in K$,
\begin{align*}
\psi(c_{i,1})=c'_{\sigma(i),1}+r_i?
\end{align*}
The condition $r\in \varphi(\Z_{p^s})$ is equivalent to the fact
that there is $z\in \Z_{p^s}$ such that $r=p^k z$. Further, each
isomorphism of the group $\Z_{p^s}$ is defined in the way: $1\mapsto
y+p^l b$ where $y\in \{1,\ldots,p^l-1\}$ and $b\in
\{0,\ldots,p^l-1\}$.

Hence, the problem reduces to the question about existing solutions
of the following system of $n$ linear equations:
\begin{align}\label{eq:cyclic}
c_{i,1}y+c_{i,1}p^l b+p^kx_i=c'_{\sigma(i),1},
\end{align}
with $2\leq i\leq n+1$ and $(n+2)$ unknowns: $y\in
\{1,\ldots,p^l-1\}$, $b\in \{0,\ldots,p^l-1\}$ and
$x_2,\ldots,x_{n+1}\in \Z_{p^s}$, for some permutation $\sigma$ of the
set $K$.

Let $B=\left(\begin{smallmatrix}
c_{2,1}& c_{2,1}p^l & p^k    & 0      &\ldots & 0\\c_{3,1}&c_{3,1}p^l&0&p^k&\ldots&0\\
\vdots &\vdots      & \vdots & \vdots &\vdots & \vdots\\
c_{n+1,1}&c_{n+1,1}p^l&0&0&\ldots&p^k\end{smallmatrix}\right)$ and
$C=\left(\begin{smallmatrix}
c'_{\sigma(2),1}\\
c'_{\sigma(3),1}\\
\vdots\\
c'_{\sigma(n+1),1}
\end{smallmatrix}\right)$.

The system \eqref{eq:cyclic} is solvable if and only if
$rk(B)=rk(B|C)$ where $rk$ denotes the rank of a matrix. Let $m=\left\lceil \frac{s}{k} \right\rceil$.
Since the sums $\mathcal{A}$ and $\mathcal{A}'$ are strictly
$(m+1)$-reductive and there is $i\in K$ such that $c_{i,1}\notin
\varphi(\Z_{p^s})$, then in the case $n\leq m$,
$rk(B)=rk(B|C)=n$
and the system
\eqref{eq:cyclic} always has a solution.

On the other hand, if $n> m$, then $rk(B)=m$. In this case, the
system has a solution if and only if there is a permutation $\sigma$ of
the set $K$ and
$c_{i,1}c'_{\sigma(j),1}-c_{j,1}c'_{\sigma(i),1}=0$ for any $i,j\in
K$.
This completes the proof.
\end{proof}

\begin{exm}
Using Theorem \ref{thm2:cyclic} it is easy to check that the quandles
$\siq(\Z_{49},43,\{1,3,4\})$ and $\siq(\Z_{49},43,\{2,5,6\})$
are not isomorphic. But their lattices of congruences are (to
compute the lattice of congruences we used \cite{FK}).

On the other hand, for each  $\Z[t,t^{-1}]$-module $\Z_{p^s}$, for a prime power $p^s$, such that for $\varphi=1-t$, $\varphi^2(\Z_{p^s})=0$ and $|\Z_{p^s}/\varphi(\Z_{p^s})|\geq 3$,
there are exactly two non-isomorphic
subdirectly irreducible medial quandles with 3 orbits, namely
$\siq(\Z_{p^s},1-\varphi,\{0,1\})$ and $\siq(\Z_{p^s},1-\varphi,\{1,c\})$ where $c\in \Z_{p^s}\setminus \varphi(\Z_{p^s})$.
\end{exm}

\section{Classification of SI 2-reductive or involutory medial quandles}\label{sec:classification}

According to Theorems~\ref{all_si} and~\ref{thm1:subdir}, there are four distinct classes of SI medial quandles.

\begin{theorem}\label{thm:all_class}
Let $Q$ be a subdirectly irreducible medial quandle. Then
$Q$ falls within one of the following four disjoint classes:
\begin{itemize}
\item $Q$ is latin (a finite or infinite cancellative Alexander quandle);
\item $Q$ is connected infinite quasi-reductive (a non-cancellative Alexander quandle);
\item $Q$ is reductive (a finite or infinite one);
\item $Q$ is non-connected infinite quasi-reductive (but non-reductive).
\end{itemize}
\end{theorem}

\begin{remark}
 We have produced examples of SI medial quandles for each of the classes.
 A family of finite SI latin quandles was given in Example \ref{ex:lat}
 and infinite SI latin quandles were presented in Example \ref{exm:Prugr1}.
 Infinite connected quasi-reductive ones were constructed in Example~\ref{exm:Prugr}.

 All non-connected quasi-reductive SI medial quandles were constructed in Theorem \ref{thm1:subdir} and Corollary \ref{cor:SI_red}.
 Examples of finite reductive SI medial quandles were described in Example~\ref{exm:sired}. Infinite
 reductive SI medial quandles are presented in Example \ref{exm:si3red}.

 Finally, infinite quasi-reductive (but non-reductive) quandles of two different types are given in Examples \ref{exm:pp} and \ref{exm:pp1}.
\end{remark}

However, in order to construct all possible SI medial quandles,
we need
the classification of subdirectly irreducible $\Z[t,t^{-1}]$-modules.
But this question is still open.
In the finite case George Bergman presented a classification (Kearnes mentions it in \cite{K91}) of the isomorphism
types of finite subdirectly irreducible modules over an arbitrary ring as duals of finite cyclic left quotient modules,
but his description is not constructive.

Nevertheless, there are some subclasses of medial quandles where we can consider $\Z$-modules instead of $\Z[t,t^{-1}]$-modules and the classification
of subdirectly irreducible $\Z$-modules is known: they are either cyclic groups~$\Z_{p^k}$ or Pr\"ufer groups~$\Z_{p^\infty}$ for a prime $p$.
Examples of such classes are $2$-reductive medial quandles since $\Z[t,t^{-1}]/(1-t)\cong \Z$, and involutory medial quandles since $\Z[t,t^{-1}]/(1+t)\cong \Z$.

According to Proposition~\ref{prop:reductive}, a medial quandle is $2$-reductive if and only if each of its orbit
is a projection quandle, i.e. an Alexander quandle $(A,1)$.
Theorem \ref{thm1:subdir} now immediately gives the known
characterization of finite SI
2-reductive medial quandles which was
presented by Romanowska and Roszkowska in~\cite{RR89}.

\begin{theorem}\cite[Theorem 3.1]{RR89}\label{thm:red_si}
A finite strictly $2$-reductive medial quandle $Q$ is subdirectly irreducible if and only if
$Q$ is isomorphic to $\siq(\Z_{p^k},1,C)$ for some $C\subseteq \Z_{p^k}$ containing at least
one generator of~$\Z_{p^k}$.
\end{theorem}

Moreover, we are now able to describe all SI 2-reductive medial quandles.

\begin{theorem}\label{thm:all2red}
 All infinite subdirectly irreducible $2$-reductive medial quandles are isomorphic to
 $\siq(\Z_{p^\infty},1,C)$ where $p$ is a prime and $C$ is an infinite subset of~$\Z_{p^\infty}$.
 There are $2^\omega$ isomorphism classes of such quandles.
\end{theorem}

\begin{proof}
 Consider~$Q$ an infinite subdirectly irreducible $2$-reductive medial quandle.
 According to Remark~\ref{rem:twoorb}, we have to build it from an infinite subdirectly
 irreducible abelian group.
 Since Pr\"ufer groups are not finitely generated, there has to be infinitely many elements in~$C$.
 Any infinite subset of~$\Z_{p^\infty}$ already generates the group.

 There are $\omega$ different groups $\Z_{p^\infty}$ and each the group~$\Z_{p^\infty}$ has $2^\omega$ subsets and hence
 there are at most $2^\omega$ different subdirectly irreducible quandles.
 According to Theorem~\ref{thm:isomorphism}, two quandles $\siq(\Z_{p^\infty},1,C)$ and $\siq(\Z_{p^\infty},1,C')$
 are isomorphic if and only if there exists an automorphism~$\psi$ of~$\Z_{p^\infty}$ that sends $C$ on $C'$.
 Consider now the set $D=\{\left[\frac 1{p^k}\right]_\sim\mid k\in\mathbb{N}\}$.  Each automorphism~$\psi$ of~$\Z_{p^\infty}$ is of the form $\psi(\left[\frac 1{p^k}\right]_\sim)=\left[\frac a{p^k}\right]_\sim$, for some~$a$ coprime to~$p$.
 Thus, $\psi(D')\cap D=\emptyset$, for any subset $D'$ of $D$ and $\psi\neq 1$.
 The set~$D$ has $2^\omega$ infinite subsets and hence there are at least $2^\omega$
 medial quandles of type $\siq(\Z_{p^\infty},1,C)$ where $C\subseteq D$.
\end{proof}

A binary algebra~$Q$ is called \emph{involutory} if $L_a^2=1$, for every $a\in Q$,
i.e., if it satisfies the identity
$$x\ast(x\ast y)\approx y.$$

It is easy to show that an Alexander quandle is involutory if and only if it is~$(A,-1)$,
for some abelian group~$A$.
Now our classification
confirms the result of Roszkowska:

\begin{theorem}\cite[Theorem 4.3]{R99b}\label{thm:sym_si}
A finite involutory medial quandle $Q$ is subdirectly
irreducible if and only if $Q$ is isomorphic to one of the following quandles:
\begin{itemize}
 \item Alexander quandle $(\Z_2,1)$,
 \item Alexander quandle $(\Z_{p^k},-1)$, for an odd prime~$p$ and $k\geq 1$,
 \item $\siq(\Z_{2^k},-1,\{1\})$ or $\siq(\Z_{2^k},-1,\{0,1\})$, for some $k\geq 1$.
\end{itemize}
\end{theorem}

\begin{proof}
 Finite SI abelian groups are $\Z_{p^k}$. The multiplication by $1-(-1)$ is surjective if and only
 if $p\neq 2$. This gives all SI latin medial quandles. For the reductive ones,
 it suffices to notice that $2\Z_{2^k}$ is a subgroup of index~$2$ and that the choice of
 transversal representatives is irrelevant, according to Lemma~\ref{lm11:subdir}.
\end{proof}

Moreover, we can even present all infinite subdirectly irreducible involutory medial quandles.

\begin{theorem}\label{th:inf_inv}
The only infinite SI involutory medial quandles are Alexander quandles $(\Z_{2^\infty},-1)$
and the quasi-reductive quandle $\siq(\Z_{2^\infty},-1,\{\left[\frac 0{2}\right]_\sim\})$.
\end{theorem}

\begin{proof}
 The only infinite subdirectly irreducible abelian group where the multiplication by~$2$ is not a bijection
 is $\Z_{2^\infty}$. Since $2\Z_{2^\infty}=\Z_{2^\infty}$, we have $|C|=1$
 and the choice of the element in~$C$ is not important, according to Lemma~\ref{lm11:subdir}.
\end{proof}

\vskip 2mm
\paragraph*{\bf Acknowledgement}
We would like to thank David Stanovsk\'y for his valuable comments and fruitful discussion, and Keith Kearnes for giving to us details of Bergman's construction. We also wish to thank the
referee for his/her comment that each SI medial quandle of quasi-affine type is connected which allowed to complete the classification of infinite SI medial quandles.

After revised version of our paper was submitted, Joseph D.Cyr \cite{C} sent to us a direct proof that each SI medial quandle of set type is quasi-reductive (which follows from his more general results). In the paper, it follows from Theorems~\ref{thm:KK} and \ref{thm:all_class}.


\begin{thebibliography}{99}

\bibitem{AG}
N. Andruskiewitsch, M. Gra\~na, \emph{From racks to pointed Hopf algebras}, Adv. Math. 178/2 (2003), 177--243.

\bibitem{Ber}
C.~Bergman, \emph{Universal algebra: Fundamentals and selected
topics}, Chapman \& Hall/CRC Press, 2011.

\bibitem{C}
J.D. Cyr, \emph{Private communication}, 2017.

\bibitem{EN}
M. Elhamdadi, S. Nelson, \emph{Quandles: An introduction to the algebra of knots}, American Mathematical Society, Providence, 2015.

\bibitem{FK}
R. Freese, E. Kiss, M. Valeriote, \emph{UACalc, A Universal Algebra
Calculator}, http://www.uacalc.org/.

\bibitem{HSV}
A. Hulpke, D. Stanovsk\'y, P. Vojt\v echovsk\'y, {\it Connected quandles and transitive groups}, J. Pure Appl. Algebra 220 (2016), 735--758.

\bibitem{Ja}
A.V. Jategaonkar,  \emph{Jacobson's conjecture and modules over fully bounded Noetherian rings}, J. Algebra 30 (1974), 103--121.

\bibitem{JPSZ}
P. Jedli\v cka, A. Pilitowska, D. Stanovsk\'y, A. Zamojska-Dzienio, \emph{The structure of medial quandles}, J. Algebra 443 (2015), 300--334.

\bibitem{JPZ17}
P. Jedli\v cka, A. Pilitowska, A. Zamojska-Dzienio, \emph{Free medial quandles}, Algebra Universalis 78 (2017), 43-54.

\bibitem{JKS}
E. Je\v r\'abek, T. Kepka, D. Stanovsk\'y, \emph{Subdirectly irreducible non-idempotent left symmetric left distributive groupoids}, Discussiones Math. - General Algebra and Appl. 25/2 (2005) 235--257.

\bibitem{Joy}
D. Joyce, \emph{Classifying invariant of knots, the knot quandle}, J. Pure Appl. Algebra,  23 (1982), 37--65.

\bibitem{J82b}
D. Joyce, \emph{Simple quandles}, J. Algebra 79 (1982), 307--318.

\bibitem{K91}
K. Kearnes, \emph{Residual bounds for varieties of modules},
Algebra Universalis 28 (1991), 448--452.

\bibitem{K95}
K. Kearnes, \emph{Semilattice modes I: the associated semiring},
Algebra Universalis 34 (1995), 220--272.

\bibitem{K99}
K. Kearnes, \emph{ Subdirectly irreducible modes}, Discuss. Math.
Algebra Stochastic Methods 19 (1999), 129--145.

\bibitem{MMT}
R. McKenzie, G. McNulty, W. Taylor, \emph{Algebras, Lattices, Varieties},
Vol.1, Wadsworth\&Brooks/Cole, Monterey, California, 1987.

\bibitem{Przyt}
J.H. Przytycki, \emph{Conway type invariants of links and Kauffman's method}, Chapter III of the book \emph{KNOTS: From combinatorics of knot diagrams to combinatorial topology based on knots}, available at {\tt http://arxiv.org/abs/1209.1592}.

\bibitem{RR89}
A. Romanowska, B. Roszkowska, \emph{Representations of $n$-cyclic
groupoids,} Algebra Universalis 26 (1989), 7--15.

\bibitem{RS}
A. Romanowska, J.D.H. Smith, \emph{Modes,} World Scientific, 2002.

\bibitem{R99b}
B. Roszkowska-Lech, {\it Subdirectly irreducible symmetric
idempotent and entropic groupoids}, Demonstratio Mathematica 32/3
(1999), 469--484.

\bibitem{St}
D. Stanovsk\'y, \emph{Subdirectly irreducible differential modes}, Internat. J. Algebra Comput. 22/4 (2012), 1250028 (18p.).

\bibitem{St08}
D. Stanovsk\'y, \emph{Subdirectly irreducible non-idempotent left distributive left quasigroups}, Commun. Algebra 36/7 (2008), 2654--2669.

\end{thebibliography}
\end{document}